\documentclass{amsart}

\usepackage[T1]{fontenc}
\usepackage{enumerate, amsmath, amsfonts, amssymb, amsthm, mathrsfs, wasysym, graphics, graphicx, xcolor, url, hyperref, hypcap, shuffle, xargs, multicol, overpic, pdflscape, multirow, hvfloat, minibox, accents, array, multido, xifthen, a4wide, ae, aecompl, blkarray, pifont, mathtools, etoolbox, dsfont}
\usepackage{marginnote}
\hypersetup{colorlinks=true, citecolor=darkblue, linkcolor=darkblue}
\usepackage[all]{xy}
\usepackage[bottom]{footmisc}
\usepackage{tikz}
\usepackage{tikz-qtree}
\usetikzlibrary{trees, decorations, decorations.markings, shapes, arrows, matrix, calc, fit, intersections, patterns, angles, snakes}
\usepackage[external]{forest}
\graphicspath{{figures/}{figures/nodes/}}
\makeatletter\def\input@path{{figures/}}\makeatother
\usepackage{caption}
\usepackage[noabbrev,capitalise]{cleveref}
\usepackage[export]{adjustbox}
\usepackage{ulem}
\usepackage{picins/picins}

\newtheorem{theorem}{Theorem}

\newtheorem{proposition}[theorem]{Proposition}
\newtheorem{lemma}[theorem]{Lemma}

\newtheorem*{theorem*}{Theorem}

\theoremstyle{definition}
\newtheorem{definition}[theorem]{Definition}
\newtheorem{example}[theorem]{Example}
\newtheorem{remark}[theorem]{Remark}

\newtheorem{problem}[theorem]{Problem}
\newtheorem{notation}[theorem]{Notation}

\crefname{notation}{Notation}{Notations}
\crefname{problem}{Problem}{Problems}
 
\newcommand{\R}{\mathbb{R}} 
\newcommand{\N}{\mathbb{N}} 
\newcommand{\HH}{\mathbb{H}} 
\renewcommand{\b}[1]{{\boldsymbol{#1}}} 

\newcommand{\set}[2]{\left\{ #1 \;\middle|\; #2 \right\}} 
\newcommand{\ssm}{\smallsetminus} 
\newcommand{\dotprod}[2]{\left\langle \, #1 \; \middle| \; #2 \, \right\rangle} 
\newcommand{\one}{\b{1}} 
\newcommand{\eqdef}{\mbox{\,\raisebox{0.2ex}{\scriptsize\ensuremath{\mathrm:}}\ensuremath{=}\,}} 

\DeclareMathOperator{\cone}{cone} 

\newcommand{\ie}{\textit{i.e.}~} 
\newcommand{\aka}{\textit{a.k.a.}~} 
\definecolor{darkblue}{rgb}{0,0,0.7} 
\definecolor{green}{RGB}{57,181,74} 
\definecolor{violet}{RGB}{147,39,143} 
\newcommand{\darkblue}{\color{darkblue}} 
\newcommand{\defn}[1]{\textsl{\darkblue #1}} 
\newcommand{\para}[1]{\bigskip\noindent\uline{#1.}} 
\newcommand{\OEIS}[1]{{\rm \href{http://oeis.org/#1}{\texttt{#1}}}}
\newcommand{\imagetop}[1]{\vtop{\null\hbox{#1}}} 

\usepackage{todonotes}

\newcommandx{\projDown}[1][1={}]{\smash{\pi_\downarrow^{#1}}} 
\newcommandx{\projUp}[1][1={}]{\smash{\pi^\uparrow_{#1}}} 

\newcommandx{\ray}[1][1=J]{\b{r}_{#1}} 
\newcommandx{\rays}[1][1=J]{\b{R}_{#1}} 
\newcommandx{\Fan}[1][1=]{\mathcal{F}_{#1}} 
\newcommand{\polytope}[1]{\mathds{#1}} 
\newcommand{\Cone}{\polytope{C}} 

\newcommandx{\pebbleColors}[1][1=\Gamma]{#1} 
\newcommandx{\pebbleColor}[1][1=\gamma]{#1} 
\newcommandx{\pebbleNumber}[2][1=S, 2=\pebbleColor]{\pi_{#2}(#1)} 
\newcommandx{\leaves}[1][1=S]{\mathsf{L}(#1)} 
\newcommandx{\leavesNumber}[1][1=S]{\ell(#1)} 
\newcommandx{\pebbleDefault}[2][1=S, 2=\pebbleColor]{\Delta_{#2}(#1)} 
\newcommandx{\balanced}[1][1=S]{\mathsf{B}(#1)} 
\newcommandx{\unbalanced}[1][1=S]{\mathsf{U}(#1)} 
\newcommandx{\pebbleTrees}[2][1=\ell,2={b,u}]{\smash{\mathcal{P}_{#1}^{#2}}} 
\newcommandx{\pebbleTreePoset}[2][1=\ell,2={b,u}]{\smash{\mathcal{P\!P}_{#1}^{#2}}} 
\newcommandx{\pebbleTreeComplex}[2][1=\ell,2={b,u}]{\smash{\mathcal{PC}_{#1}^{#2}}} 
\newcommandx{\pebbleTreeFan}[2][1=\ell,2={b,u}]{\smash{\mathcal{P\!F}_{#1}^{#2}}} 
\newcommandx{\pebbleTreeGeneratingFunction}[1][1={b,u}]{\mathfrak{P}^{#1}} 
\newcommandx{\maximalPebbleTreeGeneratingFunction}[1][1={b,u}]{\mathfrak{M}^{#1}} 
\newcommandx{\pebbleTreePolytope}[2][1=\ell,2={b,u}]{\polytope{PTP}_{#1}^{#2}} 
\tikzset{
 arrowdown/.style={postaction={decorate,decoration={
        markings,
        mark=at position .7 with {\arrow{>}}
      }}},
 arrowup/.style={postaction={decorate,decoration={
        markings,
        mark=at position .7 with {\arrow{<}}
      }}}
}
\forestset{
fairly nice empty nodes/.style={
delay={where content={}
{shape=coordinate, for siblings={anchor=north}}{}},
for tree={s sep=4mm}
}
}
\newcommandx{\pebbleTree}[3][2=7pt, 3=3pt]{
	\imagetop{
		\begin{forest}
		fairly nice empty nodes,
		for tree={l=#2, l sep=#2, inner sep=1pt, s sep=#3} #1
		\end{forest}
	}
}
\newcommand{\leaf}{\phantom{$\circ$}} 

\newcommand{\yinyang}{%
    \begin{tikzpicture}[scale=0.45]
      \draw[line width = 0.05ex,transform canvas={yshift=0.12ex}] (0,0) circle (1ex);
      \path[fill=black,transform canvas={yshift=0.12ex}] (90:1ex) arc (90:-90:0.5ex)
                        (0,0)    arc (90:270:0.5ex)
                        (0,-1ex) arc (-90:-270:1ex);
    \end{tikzpicture}}

\makeatletter
\def\l@part{\@tocline{1}{8pt}{0pc}{}{}}
\def\l@section{\@tocline{1}{4pt}{0pc}{}{}}
\makeatother
\let\oldtocpart=\tocpart
\renewcommand{\tocpart}[2]{\sc\large\oldtocpart{#1}{#2}}
\let\oldtocsection=\tocsection
\renewcommand{\tocsection}[2]{\bf\oldtocsection{#1}{#2}}
\let\oldtocsubsubsection=\tocsubsubsection
\renewcommand{\tocsubsubsection}[2]{\quad\oldtocsubsubsection{#1}{#2}}

\title{Pebble trees}

\thanks{Partially supported by the Spanish grant PID2022-137283NB-C21 of MCIN/AEI/10.13039/501100011033 / FEDER, UE, by Departament de Recerca i Universitats de la Generalitat de Catalunya (2021 SGR 00697), by the French grant CHARMS (ANR-19-CE40-0017), and by the French--Austrian project PAGCAP (ANR-21-CE48-0020 \& FWF I 5788).}

\author{Vincent Pilaud}
\address{Universitat de Barcelona}
\email{vincent.pilaud@ub.edu}
\urladdr{\url{https://www.ub.edu/comb/vincentpilaud/}}

\begin{document}

\begin{abstract}
A pebble tree is an ordered tree where each node receives some colored pebbles, in such a way that each unary node receives at least one pebble, and each subtree has either one more or as many leaves as pebbles of each color.
We show that the contraction poset on pebble trees is isomorphic to the face poset of a convex polytope called pebble tree polytope.
Beside providing intriguing generalizations of the classical permutahedra and associahedra, our motivation is that the faces of the pebble tree polytopes provide realizations as convex polytopes of all assocoipahedra constructed by K.~Poirier and T.~Tradler only as polytopal complexes.
\end{abstract}

\vspace*{-.5cm}
\maketitle

\vspace{-.5cm}
\tableofcontents
\vspace{-.7cm}


\section{Introduction}

Permutahedra and associahedra are among the most classical polytopes in algebraic combinatorics.
The $(n-1)$-dimensional permutahedron has a vertex for each permutation of~$[n] \eqdef \{1, \dots, n\}$ and an edge for each pair of permutations related by the transposition of two adjacent entries.
The $(n-1)$-dimensional associahedron has a vertex for each binary tree on~$n$ nodes and an edge for each pair of binary trees related by a rotation.
These two families of polytopes admit common generalizations explaining their similar behavior, including the permutreehedra of~\cite{PilaudPons-permutrees}, the quotientopes of~\cite{PilaudSantos-quotientopes}, and the $(m,n)$-multiplihedra of~\cite{ChapotonPilaud}.
All these polytopes are actually deformed permutahedra (defined as generalized permutahedra in~\cite{Postnikov,PostnikovReinerWilliams}), meaning that their normal fans all coarsen the braid fan.
This paper is devoted to another common generalization to the permutahedra and associahedra, which are not deformed permutahedra in general.

The combinatorics of this generalization is based on pebble trees.
A \defn{pebble tree} is an ordered tree where each node receives some colored pebbles in such a way that each unary node receives at least one pebble, and each subtree has either one more or as many leaves as pebbles of each color (see \cref{fig:pebbleTrees}).
We consider the set of pebble trees with a fixed number of leaves and fixed sets of unbalanced and balanced colors (\ie for which colors the whole pebble tree as one more or as many leaves as pebbles).
It is immediate from the definition that these pebble trees are closed by arc contractions, and our main result is that the contraction poset is the face poset of a simple convex polytope, that we call the \defn{pebble tree polytope} (see \cref{fig:pebbleTreePolytopes}).
For this, we first construct the \defn{pebble tree fan}, associating to each pebble tree a polyhedral cone in a very natural way, and then prove that this fan is the normal fan of a polytope by checking that its wall-crossing inequalities are satisfiable.
\begin{figure}
	\centerline{
		\pebbleTree{[ [ ${\circ}{\circ}{\circ}{\bullet}$ [ $\bullet$ [ ${\circ}{\bullet}$ [ \leaf ] [ \leaf ] ] [ $\circ$ [ \leaf ] ] ] [ ${\circ}{\circ}{\bullet}{\bullet}$ [ \leaf ] [ \leaf ] [ \leaf ] ] [ ${\circ}{\bullet}$ [ ${\circ}{\bullet}$ [ \leaf ] [ $\bullet$ [ \leaf ] ] ] [ \leaf ] ] ] ]}
		\quad
		\pebbleTree{[ [ ${\circ}{\circ}{\circ}{\bullet}$ [ $\bullet$ [ ${\circ}{\bullet}$ [ \leaf ] [ \leaf ] ] [ $\circ$ [ \leaf ] ] ] [ ${\circ}{\circ}{\bullet}{\bullet}$ [ \leaf ] [ \leaf ] [ \leaf ] ] [ ${\circ}{\circ}{\bullet}{\bullet}$ [ \leaf ] [ $\bullet$ [ \leaf ] ] [ \leaf ] ] ] ]}
		\quad
		\pebbleTree{[ [ [ $\circ$ [ [ $\circ$ [ \leaf ] ] [ $\bullet$ [ \leaf ] ] ] ] [ $\circ$ [ $\bullet$ [ \leaf ] ] ] ] ]}
		\quad
		\pebbleTree{[ [ $\circ$ [ [ [ $\circ$ [ \leaf ] ] [ $\bullet$ [ \leaf ] ] ] [ $\circ$ [ $\bullet$ [ \leaf ] ] ] ] ] ]}
	}
	\caption{Some $\circ$-balanced and $\bullet$-unbalanced $\{\circ, \bullet\}$-pebble trees. The first two are related by a contraction (\cref{subsec:pebbleTreePoset}), while the last two are related by a flip (\cref{subsec:pebbleTreeFlipGraph}).}
	\label{fig:pebbleTrees}
\end{figure}

Our construction recovers the combinatorics and geometry of the permutahedra and associahedra in two degenerate situations.
On the one hand, pebble trees with a single leaf can be seen as ordered partitions of their balanced colors, and the pebble tree polytope is the permutahedron.
On the other hand, pebble trees with no pebbles are Schr\"oder trees, and the pebble tree polytope is the associahedron.
But the special situation which motivated this paper is when there is exactly one color of pebbles.
The pebble trees are then in bijection with some specific oriented planar trees considered by K.~Poirier and T.~Tradler in~\cite{PoirierTradler} for the combinatorics of algebraic structures endowed with a binary product and a co-inner product.
These structures are closely connected to the $V_\infty$-algebras of T.~Tradler and M.~Zeinalian~\cite{TradlerZeinalian} that arose in a tentative algebraic model for string topology operations defined by M.~Chas and D.~Sullivan \cite{ChasSullivan}.
It is proved in~\cite{PoirierTradler} that the contraction posets on these oriented planar trees are face lattices of the \defn{assocoipahedra}, which are polytopal complexes refining the boundary complex of the Cartesian product of an associahedron with a simplex.
We prove here that all assocoipahedra can actually be realized as convex polytopes using faces of pebble tree polytopes (see \cref{fig:assocoipahedra}).

The paper is organized as follows.
\cref{sec:pebbleTreeCombinatorics} is devoted to the combinatorics of pebble trees.
In \cref{subsec:pebbleTrees}, we provide more precise definitions and notations for pebble trees, we introduce some natural maps between families of pebble trees, and we give the precise bijection with the oriented planar trees of~\cite{PoirierTradler}.
We introduce in \cref{subsec:pebbleTreePoset} the pebble tree contraction poset, prove in \cref{subsec:pebbleTreeComplex} that it is the face poset of a pseudomanifold called the pebble tree complex, and discuss in \cref{subsec:pebbleTreeFlipGraph} the adjacency graph of this complex called the pebble tree flip graph.
\cref{sec:pebbleTreeGeometry} is devoted to the geometry of pebble trees.
After quickly reminding some geometric preliminaries in \cref{subsec:geometricPreliminaries}, we construct the pebble tree fan in \cref{subsec:pebbleTreeFan} and the pebble tree polytope in \cref{subsec:pebbleTreePolytope}.
Finally, \cref{sec:pebbleTreeNumerology} is devoted to the numerology of pebble trees.
We compute the generating functions of the maximal pebble trees in \cref{subsec:enumerationMPT} and of all the pebble trees in \cref{subsec:enumerationPT}, and gather explicit expansions of these generating functions in \cref{subsec:generatingFunctions}.
While the methods are standard computations based on generatingfunctionology~\cite{FlajoletSedgewick}, the results reveal a few surprises.


\section{Pebble tree combinatorics}
\label{sec:pebbleTreeCombinatorics}

In this section, we define pebble trees (\cref{subsec:pebbleTrees}) and study the pebble tree contraction poset (\cref{subsec:pebbleTreePoset}), the pebble tree complex (\cref{subsec:pebbleTreeComplex}), and the pebble tree flip graph (\cref{subsec:pebbleTreeFlipGraph}).


\subsection{Pebble trees}
\label{subsec:pebbleTrees}

Recall that an \defn{ordered tree} is either a \defn{leaf} or a \defn{node} with an ordered list of subtrees.
These subtrees are the \defn{children} of the node, and this node is the \defn{parent} of these subtrees.
As we only consider ordered trees, we omit the adjective ordered and only say tree.
For a node~$n$ in a tree~$T$, we denote by~$T_n$ the subtree of~$T$ rooted at~$n$.
For a subtree~$S$, we denote by~$\leaves$ the set of leaves of~$S$.

\pagebreak

\begin{definition}
\label{def:pebbleTree}
For a finite set of colors~$\pebbleColors$, a \defn{$\pebbleColors$-pebble tree} is a tree with pebbles colored by~$\pebbleColors$ placed on its nodes such that
\begin{enumerate}[(i)]
\item each leaf receives no pebble, each node with a single child receives at least one pebble, and each node with at least two children receives arbitrary many pebbles (possibly none), \label{cond:pebbleNodeResctrictions}
\item for each subtree~$S$ and each color~$\pebbleColor \in \pebbleColors$, the number of leaves minus the number of pebbles of color~$\pebbleColor$ in~$S$ is either~$0$ or~$1$. \label{cond:localBalanced}
\end{enumerate}
\end{definition}

\begin{example}
\label{exm:pebbleTrees}
Two classical combinatorial objects are extreme examples of pebble trees:
\begin{itemize}
\item pebble trees with a single leaf can be seen as ordered partitions of their pebble colors,
\item pebble trees with no pebbles are Schr\"oder trees (where each node has either none or at least two children).
\end{itemize}
Some more generic examples of pebble trees are illustrated in \cref{fig:pebbleTrees}.
\end{example}

\begin{notation}
We call \defn{$\pebbleColor$-pebbles} the pebbles of color~$\pebbleColor$.
We call \defn{$\pebbleColor$-pebble default} of a subtree~$S$ the difference~$\pebbleDefault$ between the number of leaves and the number of $\pebbleColor$-pebbles of~$S$.
We say that $S$ is \defn{$\pebbleColor$-balanced} (resp.~\defn{$\pebbleColor$-unbalanced}) if~$\pebbleDefault = 0$ (resp.~$\pebbleDefault = 1$).
We denote by~$\balanced \eqdef \set{\pebbleColor \in \pebbleColors}{\pebbleDefault = 0}$ (resp.~$\unbalanced \eqdef \set{\pebbleColor \in \pebbleColors}{\pebbleDefault = 1}$) the set of colors~$\pebbleColor \in \pebbleColors$ for which~$S$ is~$\pebbleColor$-balanced (resp.~$\pebbleColor$-unbalanced).
\end{notation}

\begin{notation}
We denote by~$\pebbleTrees[L][\pebbleColors]$ the set of all pebble trees with leaves~$L$ and pebble colors~$\pebbleColors$, and by~$\pebbleTrees[L][B,U]$ the subset of $B$-balanced and $U$-unbalanced pebble trees of~$\pebbleTrees[L][\pebbleColors]$ for any~${B \sqcup U = \pebbleColors}$.
For~$\ell, b, u \in \N$, we define~$\pebbleTrees[\ell][b,u]$ as~$\pebbleTrees[{[\ell]}][{[b],[b+1,b+u]}]$.
Note that~$\pebbleTrees[L][B,U]$ is isomorphic to~$\pebbleTrees[\ell][b,u]$ for arbitrary~$L, B, U$ with~$|L| = \ell$, $|B| = b$ and~$|U| = u$.
It is however convenient to keep the notation~$\pebbleTrees[L][B,U]$ to define certain operations on pebble trees (see \cref{def:mirroringMap,def:balancingMap,def:insertingMap,def:rerootingMap}) and for recursive decompositions of the pebble trees (see \cref{sec:pebbleTreeNumerology}).
\end{notation}

\begin{remark}
Some immediate consequences of~\cref{def:pebbleTree}:
\begin{itemize}
\item $\pebbleTrees[L][\pebbleColors]$ is finite for any~$L$ and~$\pebbleColors$, thus $\pebbleTrees$ is finite for any~${\ell, b, u \in \N}$.
\item The number of $\pebbleColor$-pebbles at a node~$p$ with children~$c_1, \dots, c_k$ is $\big( \sum_{i = 1}^k \pebbleDefault[T_{c_i}] \big) - \pebbleDefault[T_p]$.
Hence, the number of $\pebbleColor$-unbalanced children of~$p$ is the number of $\pebbleColor$-pebbles at~$p$, plus~$1$ if~$p$ is $\pebbleColor$-unbalanced.
\item Each unary node has at least one pebble, and at most one of each color.
\item There is no consecutive chain formed by $|\pebbleColors| + 1$ unary nodes
\end{itemize}
\end{remark}

\enlargethispage{.3cm}
We now define five natural maps between pebble trees (see \cref{fig:balancingInsertingMirroring,fig:rerooting}), that will induce isomorphisms in \cref{prop:pebbleTreePosetOperations,prop:pebbleTreeComplexOperations,prop:pebbleTreePolytopeOperations}.
In \cref{def:insertingMap,def:uprootingMap}, we call $\pebbleColor$-leaf the only pebble tree of~$\pebbleTrees[1][\{\pebbleColor\},\varnothing]$, \ie whose root has a single pebble of color~$\pebbleColor$ and a single child which is a leaf.

\begin{figure}
	\centerline{
		\begin{tabular}{cccc}
			\pebbleTree{[ [ ${\circ}{\circ}{\circ}{\bullet}$ [ $\bullet$ [ ${\circ}{\bullet}$ [ \leaf ] [ \leaf ] ] [ $\circ$ [ \leaf ] ] ] [ ${\circ}{\circ}{\bullet}{\bullet}$ [ \leaf ] [ \leaf ] [ \leaf ] ] [ ${\circ}{\circ}{\bullet}{\bullet}$ [ \leaf ] [ $\bullet$ [ \leaf ] ] [ \leaf ] ] ] ]} &
			\pebbleTree{[ [ ${\circ}{\circ}{\circ}{\bullet}$ [ ${\circ}{\circ}{\bullet}{\bullet}$ [ \leaf ] [ $\bullet$ [ \leaf ] ] [ \leaf ] ] [ ${\circ}{\circ}{\bullet}{\bullet}$ [ \leaf ] [ \leaf ] [ \leaf ] ] [ $\bullet$ [ $\circ$ [ \leaf ] ] [ ${\circ}{\bullet}$ [ \leaf ] [ \leaf ] ] ] ] ]} &
			\pebbleTree{[ [ $\bullet$ [ ${\circ}{\circ}{\circ}{\bullet}$ [ $\bullet$ [ ${\circ}{\bullet}$ [ \leaf ] [ \leaf ] ] [ $\circ$ [ \leaf ] ] ] [ ${\circ}{\circ}{\bullet}{\bullet}$ [ \leaf ] [ \leaf ] [ \leaf ] ] [ ${\circ}{\circ}{\bullet}{\bullet}$ [ \leaf ] [ $\bullet$ [ \leaf ] ] [ \leaf ] ] ] ] ]} &
			\pebbleTree{[ [ ${\circ}{\circ}{\circ}{\bullet}$ [ $\bullet$ [ ${\circ}{\bullet}$ [ $\star$ [ \leaf ] ] [ $\star$ [ \leaf ] ] ] [ $\circ$ [ $\star$ [ \leaf ] ] ] ] [ ${\circ}{\circ}{\bullet}{\bullet}$ [ $\star$ [ \leaf ] ] [ $\star$ [ \leaf ] ] [ $\star$ [ \leaf ] ] ] [ ${\circ}{\circ}{\bullet}{\bullet}$ [ $\star$ [ \leaf ] ] [ $\bullet$ [ $\star$ [ \leaf ] ] ] [ $\star$ [ \leaf ] ] ] ] ]}
			\\[-.2cm]
			$T$ &
			$\mu(T)$ &
			$\beta_\bullet(T)$ &
			$\iota_\star(T)$
		\end{tabular}
	}
	\caption{A $\circ$-balanced and $\bullet$-unbalanced $\{\circ, \bullet\}$-pebble tree (left), the $\circ$-balanced and \mbox{$\bullet$-unba}\-lanced $\{\circ, \bullet\}$-pebble tree obtained by mirroring it (middle left), the $\{\circ,\bullet\}$-balanced \mbox{$\{\circ, \bullet\}$-pebble} tree obtained by $\bullet$-balancing it (middle right), and the $\{\circ,\star\}$-balanced and $\bullet$-unbalanced \mbox{$\{\circ, \bullet, \star\}$-pebble} tree obtained by $\star$-inserting it (right).}
	\label{fig:balancingInsertingMirroring}
\end{figure}

\begin{figure}
	\centerline{
		\begin{tabular}{cccc@{\qquad\qquad}cc}
			\pebbleTree{[ $r$ [ $\circ$ [ [ $\bullet$ [ $\circ$ [ $x$ ] ] [ $\bullet$ [ $y$ ] ] ] [ ${\circ}{\bullet}$ [ $z$ ] ] ] ] ]} &
			\pebbleTree{[ $x$ [ $\circ$ [ $\bullet$ [ $\bullet$ [ $y$ ] ] [ [ ${\circ}{\bullet}$ [ $z$ ] ] [ $\circ$ [ $r$ ] ] ] ] ] ]} &
			\pebbleTree{[ $y$ [ $\bullet$ [ $\bullet$ [ [ ${\circ}{\bullet}$ [ $z$ ] ] [ $\circ$ [ $r$ ] ] ] [ $\circ$ [ $x$ ] ] ] ] ]} &
			\pebbleTree{[ $z$ [ ${\circ}{\bullet}$ [ [ $\circ$ [ $r$ ] ] [ $\bullet$ [ $\circ$ [ $x$ ] ] [ $\bullet$ [ $y$ ] ] ] ] ] ]} &
			\pebbleTree{[ $r$ [ $\bullet$ [ $\circ$ [ $x$ ] ] [ $y$ ] [ ${\circ}{\bullet}$ [ $z$ ] ] ] ]} &
			\pebbleTree{[ $z$ [ ${\circ}{\bullet}$ [ ${\star}{\bullet}$ [ $\circ$ [ $x$ ] ] [ $\star$ [ $y$ ] ] ] ] ]}
			\\[.2cm]
			$T$ &
			$\rho_x(T)$ &
			$\rho_y(T)$ &
			$\rho_z(T)$ &
			$T$ &
			$\theta_\star(T)$
		\end{tabular}
	}
	\caption{A fully balanced $\{\circ, \bullet\}$-pebble tree and the fully balanced $\{\circ, \bullet\}$-pebble trees obtained by rerooting it at leaves $x$, $y$ and~$z$ respectively (left). A fully unbalanced $\{\circ, \bullet\}$-pebble tree and the fully balanced $\{\circ,\bullet,\star\}$-pebble tree obtained by $\star$-uprooting it (right).}
	\label{fig:rerooting}
\end{figure}

\begin{definition}
\label{def:mirroringMap}
The \defn{mirroring map} sends a pebble tree~$T$ of~$\pebbleTrees[L][B,U]$ to the pebble tree~$\mu(T)$ of~$\pebbleTrees[L][B,U]$ obtained by a vertical symmetry of the tree, meaning that~$\mu(T)$ is defined inductively by
\begin{itemize}
\item if~$T$ is just a leaf, then~$\mu(T)$ is just a leaf,
\item if~$T$ is a node with pebbles~$P$ and children~$C_1, \dots, C_j$, then~$\mu(T)$ is a node with pebbles~$P$ and children~$\mu(C_j), \dots, \mu(C_1)$.
\end{itemize}
\end{definition}

\begin{definition}
\label{def:balancingMap}
If~$\pebbleColor \in U$, the \defn{$\pebbleColor$-balancing map} sends a pebble tree~$T$ of~$\pebbleTrees[L][B,U]$ to the pebble tree~$\beta_{\pebbleColor}(T)$ of~$\pebbleTrees[L][B \cup \{\pebbleColor\}, U \ssm \{\pebbleColor\}]$ whose root has a single pebble of color~$\pebbleColor$ and a single child~$T$.
\end{definition}

\begin{definition}
\label{def:insertingMap}
If~$\pebbleColor \notin B \cup U$, the \defn{$\pebbleColor$-inserting map} sends a pebble tree~$T$ of~$\pebbleTrees[L][B,U]$ to the pebble tree~$\iota_{\pebbleColor}(T)$ of~$\pebbleTrees[L][B \cup \{\pebbleColor\}, U]$ obtained from~$T$ by replacing each leaf by a $\pebbleColor$-leaf.
\end{definition}

\begin{definition}
\label{def:rerootingMap}
If~$x \in L$ and~$U = \varnothing$, the \defn{$x$-rerooting map} sends a pebble tree~$T$ of~$\pebbleTrees[L][B,\varnothing]$ to the pebble tree~$\rho_x(T) \in \pebbleTrees[L][B,\varnothing]$ obtained by hanging~$T$ from its leaf~$x$, \ie defined inductively~by
\begin{itemize}
\item if~$T$ is just the leaf~$x$, then~$\rho_x(T)$ is just a leaf denoted~$r$,
\item if~$T$ is a node with pebbles~$P$ and children~$C_1, \dots, C_j$ and~$i \in [j]$ is such that~$x \in \leaves[C_i]$, then~$\rho_x(T)$ is obtained by replacing the leaf~$r$ of~$\rho_x(C_i)$ by a node with pebbles~$P$ and children~$C_{i+1}, \dots, C_j, r, C_1, \dots, C_{i-1}$.
\end{itemize}
\end{definition}

\begin{definition}
\label{def:uprootingMap}
If~$\ell > 1$, $B = \varnothing$ and~$\pebbleColor \notin U$, the \defn{$\pebbleColor$-uprooting map} sends a pebble tree~$T$ of~$\pebbleTrees[{[\ell]}][\varnothing, U]$ to the pebble tree~$\theta_{\pebbleColor}(T)$ of~$\pebbleTrees[{[\ell-1]}][U \cup \{\pebbleColor\}, \varnothing]$ obtained from~$T$ by first hanging~$T$ from its rightmost leaf, then deleting the leftmost leaf and placing a $\pebbleColor$-pebble at its parent, and finally replacing all remaining leaves except the first by a $\pebbleColor$-leaf.\end{definition}

Finally, our next three remarks connect pebble trees with other relevant families of trees.

\begin{remark}
\label{rem:bijectionPebbleTreesOrientedTrees1}
Consider a word~$\alpha \in \{\textsc{i}, \textsc{o}\}^{\ell+1}$ starting with~$\textsc{o}$ (here, $\textsc{i}$ and $\textsc{o}$ stand for incoming and outgoing).
An \defn{$\alpha$-tree} is a rooted oriented planar tree such that
\begin{itemize}
\item labeling the external arrows counterclockwise starting from the root, the $i$th arrow is incoming if~$\alpha_i = \textsc{i}$ and outgoing if~$\alpha_i = \textsc{o}$,
\item each internal node has at least one outgoing arrow,
\item there is no node with precisely one incoming and one outgoing arrow.
\end{itemize}
These trees arise in the combinatorics of algebras endowed with a binary product and a co-inner product.
They are studied in details in~\cite{PoirierTradler}.
It turns out that they can be understood from pebble trees.

First, as illustrated in \cref{fig:orientation}, there are simple bijections between the pebble trees of~$\pebbleTrees[\ell][1,0]$ and the $\textsc{o}^{\ell+1}$-trees:
\begin{itemize}
\item Starting from a pebble tree~$T \in \pebbleTrees[\ell][1,0]$, orient each arc~$(p,c)$ of~$T$ from~$c$ to~$p$ if $c$ is balanced, and from $p$ to $c$ if $c$ is unbalanced, and forget all pebbles.
\item Starting from a $\textsc{o}^{\ell+1}$-tree, place at each node one less pebbles than its outdegree, and forget the orientations.
\end{itemize}
In the present paper, we prefer our interpretation as pebble trees as it enables us to consider several pebble colors simultaneously.

Consider now an arbitrary signature~$\alpha$.
Although not explicit in~\cite{PoirierTradler}, there is a clear map from $\alpha$-trees to~$\textsc{o}^{\ell+1}$-trees, which consists in replacing each incoming external arrow (like~$\uparrow$) by a node with a pair of outgoing arrows (like~$\updownarrow$).
This leads to a bijection between the $\alpha$-trees and the pebble trees of~$\pebbleTrees[\ell][1,0]$ where the parent of the $i$th leaf is a unary node marked with a pebble.
\begin{figure}
	\centerline{
		\begin{tabular}{c@{\,}c@{\,}c@{\,}c@{\,}c@{\,}c@{\,}c@{\,}c}
			\pebbleTree{[ [ [ $\bullet$ [ [ $\bullet$ [ \leaf ] ] [ \leaf ] ] ] [ $\bullet$ [ \leaf ] ] ] ]} &
			\pebbleTree{[ [ $\bullet$ [ [ $\bullet$ [ \leaf ] ] [ \leaf ] ] [ $\bullet$ [ \leaf ] ] ] ]} &
			\pebbleTree{[ [ $\bullet$ [ $\bullet$ [ [ $\bullet$ [ \leaf ] ] [ \leaf ] ] ] [ \leaf ] ] ]} &
			\pebbleTree{[ [ [ $\bullet$ [ $\bullet$ [ \leaf ] ] [ \leaf ] ] [ $\bullet$ [ \leaf ] ] ] ]} &
			\pebbleTree{[ [ [ $\bullet$ [ $\bullet$ [ \leaf ] [ \leaf ] ] ] [ $\bullet$ [ \leaf ] ] ] ]} &
			\pebbleTree{[ [ $\bullet$ [ $\bullet$ [ \leaf ] [ \leaf ] ] [ $\bullet$ [ \leaf ] ] ] ]} &
			\pebbleTree{[ [ ${\bullet}{\bullet}$ [ $\bullet$ [ \leaf ] [ \leaf ] ] [ \leaf ] ] ]} &
			\pebbleTree{[ [ ${\bullet}{\bullet}{\bullet}$ [ \leaf ] [ \leaf ] [ \leaf ] ] ]}
			\\[-.3cm]
			\pebbleTree{ [ [ {}, edge=arrowup [ {}, edge=arrowup [ {}, edge=arrowdown [ {}, edge=arrowup [ \leaf, edge=arrowdown ] ] [ \leaf, edge=arrowdown ] ] ] [ {}, edge=arrowup [ \leaf, edge=arrowdown ] ] ] ]}[9pt][4pt] &
			\pebbleTree{[ [ {}, edge=arrowup [ {}, edge=arrowdown [ {}, edge=arrowup [ \leaf, edge=arrowdown ] ] [ \leaf, edge=arrowdown ] ] [ {}, edge=arrowup [ \leaf, edge=arrowdown ] ] ] ]}[9pt][4pt] &
			\pebbleTree{[ [ {}, edge=arrowup [ {}, edge=arrowup [ {}, edge=arrowdown [ {}, edge=arrowup [ \leaf, edge=arrowdown ] ] [ \leaf, edge=arrowdown ] ] ] [ \leaf, edge=arrowdown ] ] ]}[9pt][4pt] &
			\pebbleTree{[ [ {}, edge=arrowup [ {}, edge=arrowup [ {}, edge=arrowup [ \leaf, edge=arrowdown ] ] [ \leaf, edge=arrowdown ] ] [ {}, edge=arrowup [ \leaf, edge=arrowdown ] ] ] ]}[9pt][4pt] &
			\pebbleTree{[ [ {}, edge=arrowup [ {}, edge=arrowup [ {}, edge=arrowdown [ \leaf, edge=arrowdown ] [ \leaf, edge=arrowdown ] ] ] [ {}, edge=arrowup [ \leaf, edge=arrowdown ] ] ] ]}[9pt][4pt] &
			\pebbleTree{[ [ {}, edge=arrowup [ {}, edge=arrowdown [ \leaf, edge=arrowdown ] [ \leaf, edge=arrowdown ] ] [ {}, edge=arrowup [ \leaf, edge=arrowdown ] ] ] ]}[9pt][4pt] &
			\pebbleTree{[ [ {}, edge=arrowup [ {}, edge=arrowdown [ \leaf, edge=arrowdown ] [ \leaf, edge=arrowdown ] ] [ \leaf, edge=arrowdown ] ] ]}[9pt][4pt] &
			\pebbleTree{[ [ {}, edge=arrowup [ \leaf, edge=arrowdown ] [ \leaf, edge=arrowdown ] [ \leaf, edge=arrowdown ] ] ]}[9pt][4pt]
		\end{tabular}
	}
	\caption{Examples of the bijection between the pebble trees of~$\pebbleTrees[3][1,0]$ and the $\textsc{o}^4$-trees.}
	\label{fig:orientation}
\end{figure}
\end{remark}

\begin{remark}
\label{rem:CoriJacquardSchaeffer}
The reader familiar with the bijective combinatorics of planar maps might also see some connections with the $\beta$-trees of~\cite{JacquardSchaeffer,CoriSchaeffer}.
Indeed, labeling each node~$p$ of a pebble tree by the pebble default~$\pebbleDefault[T_p]$, we obtain a $\beta(1,1)$-tree.
However, this map is clearly injective but not surjective, and the additional condition given by the pebble trees is unclear on $\beta(1,1)$-trees.
\end{remark}

\begin{remark}
\label{rem:multiplihedra}
There is also a natural map from the pebble trees of~$\pebbleTrees[\ell][1,0]$ to the painted trees corresponding to the faces of the multiplihedron~\cite{Stasheff-HSpaces, Forcey-multiplihedra, ChapotonPilaud}.
Namely, we can just forget all pebbles which have a pebble on the path to the root of the pebble tree to obtain a painted tree.
\end{remark}


\subsection{Pebble tree contraction poset}
\label{subsec:pebbleTreePoset}

We now define the contraction poset on pebble trees, and will see that it is the face poset of a simplicial complex in \cref{subsec:pebbleTreeComplex} and of a polytope in \cref{subsec:pebbleTreePolytope}.

\begin{figure}[p]
	\centerline{
		\begin{tikzpicture}[yscale=2.7]
			\begin{scope}[xscale=.9]
      			\node (v1) at (-9,4.1) {\pebbleTree{[ [ [ \leaf ] [ $\bullet$ [ [ $\bullet$ [ \leaf ] ] [ \leaf ] ] ] ] ]}};
      			\node (v2) at (-7,4.1) {\pebbleTree{[ [ [ $\bullet$ [ \leaf ] ] [ [ $\bullet$ [ \leaf ] ] [ \leaf ] ] ] ]}};
      			\node (v3) at (-5,4.1) {\pebbleTree{[ [ [ $\bullet$ [ \leaf ] ] [ [ \leaf ] [ $\bullet$ [ \leaf ] ] ] ] ]}};
      			\node (v4) at (-3,4.1) {\pebbleTree{[ [ [ \leaf ] [ $\bullet$ [ [ \leaf ] [ $\bullet$ [ \leaf ] ] ] ] ] ]}};
      			\node (v5) at (-1,4.1) {\pebbleTree{[ [ [ \leaf ] [ [ $\bullet$ [ \leaf ] ] [ $\bullet$ [ \leaf ] ] ] ] ]}};
      			\node (v6) at (1,4.1) {\pebbleTree{[ [ [ [ $\bullet$ [ \leaf ] ] [ $\bullet$ [ \leaf ] ] ] [ \leaf ] ] ]}};
      			\node (v7) at (3,4.1) {\pebbleTree{[ [ [ $\bullet$ [ [ $\bullet$ [ \leaf ] ] [ \leaf ] ] ] [ \leaf ] ] ]}};
      			\node (v8) at (5,4.1) {\pebbleTree{[ [ [ [ $\bullet$ [ \leaf ] ] [ \leaf ] ] [ $\bullet$ [ \leaf ] ] ] ]}};
      			\node (v9) at (7,4.1) {\pebbleTree{[ [ [ [ \leaf ] [ $\bullet$ [ \leaf ] ] ] [ $\bullet$ [ \leaf ] ] ] ]}};
      			\node (v10) at (9,4.1) {\pebbleTree{[ [ [ $\bullet$ [ [ \leaf ] [ $\bullet$ [ \leaf ] ] ] ] [ \leaf ] ] ]}};
			\end{scope}
			\begin{scope}[xscale=1.2]
      			\node (e1) at (-7,3) {\pebbleTree{[ [ $\bullet$ [ \leaf ] [ [ $\bullet$ [ \leaf ] ] [ \leaf ] ] ] ]}};
      			\node (e2) at (-6,3) {\pebbleTree{[ [ [ $\bullet$ [ \leaf ] ] [ $\bullet$ [ \leaf ] [ \leaf ] ] ] ]}};
      			\node (e3) at (-5,3) {\pebbleTree{[ [ [ \leaf ] [ $\bullet$ [ $\bullet$ [ \leaf ] [ \leaf ] ] ] ] ]}};
      			\node (e4) at (-4,3) {\pebbleTree{[ [ $\bullet$ [ \leaf ] [ [ \leaf ] [ $\bullet$ [ \leaf ] ] ] ] ]}};
      			\node (e5) at (-3,3) {\pebbleTree{[ [ [ \leaf ] [ $\bullet$ [ $\bullet$ [ \leaf ] ] [ \leaf ] ] ] ]}};
      			\node (e6) at (-2,3) {\pebbleTree{[ [ [ \leaf ] [ $\bullet$ [ \leaf ] [ $\bullet$ [ \leaf ] ] ] ] ]}};
      			\node (e7) at (-1,3) {\pebbleTree{[ [ [ $\bullet$ [ \leaf ] ] [ $\bullet$ [ \leaf ] ] [ \leaf ] ] ]}};
      			\node (e8) at (0,3) {\pebbleTree{[ [ [ $\bullet$ [ \leaf ] ] [ \leaf ] [ $\bullet$ [ \leaf ] ] ] ]}};
      			\node (e9) at (1,3) {\pebbleTree{[ [ [ \leaf ] [ $\bullet$ [ \leaf ] ] [ $\bullet$ [ \leaf ] ] ] ]}};
      			\node (e10) at (2,3) {\pebbleTree{[ [ [ $\bullet$ [ $\bullet$ [ \leaf ] ] [ \leaf ] ] [ \leaf ] ] ]}};
      			\node (e11) at (3,3) {\pebbleTree{[ [ [ $\bullet$ [ \leaf ] [ $\bullet$ [ \leaf ] ] ] [ \leaf ] ] ]}};
      			\node (e12) at (4,3) {\pebbleTree{[ [ $\bullet$ [ [ $\bullet$ [ \leaf ] ] [ \leaf ] ] [ \leaf ] ] ]}};
      			\node (e13) at (5,3) {\pebbleTree{[ [ [ $\bullet$ [ $\bullet$ [ \leaf ] [ \leaf ] ] ] [ \leaf ] ] ]}};
      			\node (e14) at (6,3) {\pebbleTree{[ [ [ $\bullet$ [ \leaf ] [ \leaf ] ] [ $\bullet$ [ \leaf ] ] ] ] ]}};
      			\node (e15) at (7,3) {\pebbleTree{[ [ $\bullet$ [ [ \leaf ] [ $\bullet$ [ \leaf ] ] ] [ \leaf ] ] ]}};
			\end{scope}
			\begin{scope}[xscale=1.8]
      			\node (f1) at (-3,2) {\pebbleTree{[ [ $\bullet$ [ \leaf ] [ $\bullet$ [ \leaf ] [ \leaf ] ] ] ]}};
      			\node (f2) at (-2,2) {\pebbleTree{[ [ [ \leaf ] [ ${\bullet}{\bullet}$ [ \leaf ] [ \leaf ] ] ] ]}};
      			\node (f3) at (-1,2) {\pebbleTree{[ [ $\bullet$ [ $\bullet$ [ \leaf ] ] [ \leaf ] [ \leaf ] ] ]}};
      			\node (f4) at (0,2) {\pebbleTree{[ [ $\bullet$ [ \leaf ] [ $\bullet$ [ \leaf ] ] [ \leaf ] ] ]}};
      			\node (f5) at (1,2) {\pebbleTree{[ [ $\bullet$ [ \leaf ] [ \leaf ] [ $\bullet$ [ \leaf ] ] ] ]}};
      			\node (f6) at (2,2) {\pebbleTree{[ [ [ ${\bullet}{\bullet}$ [ \leaf ] [ \leaf ] ] [ \leaf ] ] ]}};
      			\node (f7) at (3,2) {\pebbleTree{[ [ $\bullet$ [ \leaf ] [ $\bullet$ [ \leaf ] [ \leaf ] ] ] ]}};
			\end{scope}
			\begin{scope}[xshift=-1]
      			\node (p) at (0,1.2) {\pebbleTree{[ [ ${\bullet}{\bullet}$ [ \leaf ] [ \leaf ] [ \leaf ] ] ]}};
			\end{scope}
			\draw[color=gray] (p.north) -- (f1.south);
			\draw[color=gray] (p.north) -- (f2.south);
			\draw[color=gray] (p.north) -- (f3.south);
			\draw[color=gray] (p.north) -- (f4.south);
			\draw[color=gray] (p.north) -- (f5.south);
			\draw[color=gray] (p.north) -- (f6.south);
			\draw[color=gray] (p.north) -- (f7.south);
			\draw[color=gray] (f1.north) -- (e1.south);
			\draw[color=gray] (f1.north) -- (e2.south);
			\draw[color=gray] (f1.north) -- (e3.south);
			\draw[color=gray] (f1.north) -- (e4.south);
			\draw[color=gray] (f2.north) -- (e3.south);
			\draw[color=gray] (f2.north) -- (e5.south);
			\draw[color=gray] (f2.north) -- (e6.south);
			\draw[color=gray] (f3.north) -- (e2.south);
			\draw[color=gray] (f3.north) -- (e7.south);
			\draw[color=gray] (f3.north) -- (e8.south);
			\draw[color=gray] (f3.north) -- (e10.south);
			\draw[color=gray] (f3.north) -- (e12.south);
			\draw[color=gray] (f4.north) -- (e1.south);
			\draw[color=gray] (f4.north) -- (e5.south);
			\draw[color=gray] (f4.north) -- (e7.south);
			\draw[color=gray] (f4.north) -- (e9.south);
			\draw[color=gray] (f4.north) -- (e11.south);
			\draw[color=gray] (f4.north) -- (e15.south);
			\draw[color=gray] (f5.north) -- (e4.south);
			\draw[color=gray] (f5.north) -- (e6.south);
			\draw[color=gray] (f5.north) -- (e8.south);
			\draw[color=gray] (f5.north) -- (e9.south);
			\draw[color=gray] (f5.north) -- (e14.south);
			\draw[color=gray] (f6.north) -- (e10.south);
			\draw[color=gray] (f6.north) -- (e11.south);
			\draw[color=gray] (f6.north) -- (e13.south);
			\draw[color=gray] (f7.north) -- (e12.south);
			\draw[color=gray] (f7.north) -- (e13.south);
			\draw[color=gray] (f7.north) -- (e14.south);
			\draw[color=gray] (f7.north) -- (e15.south);
			\draw[color=gray] (e1.north) -- (v2.south);
			\draw[color=gray] (e1.north) -- (v1.south);
			\draw[color=gray] (e2.north) -- (v2.south);
			\draw[color=gray] (e2.north) -- (v3.south);
			\draw[color=gray] (e3.north) -- (v1.south);
			\draw[color=gray] (e3.north) -- (v4.south);
			\draw[color=gray] (e4.north) -- (v3.south);
			\draw[color=gray] (e4.north) -- (v4.south);
			\draw[color=gray] (e5.north) -- (v1.south);
			\draw[color=gray] (e5.north) -- (v5.south);
			\draw[color=gray] (e6.north) -- (v4.south);
			\draw[color=gray] (e6.north) -- (v5.south);
			\draw[color=gray] (e7.north) -- (v2.south);
			\draw[color=gray] (e7.north) -- (v6.south);
			\draw[color=gray] (e8.north) -- (v3.south);
			\draw[color=gray] (e8.north) -- (v8.south);
			\draw[color=gray] (e9.north) -- (v5.south);
			\draw[color=gray] (e9.north) -- (v9.south);
			\draw[color=gray] (e10.north) -- (v6.south);
			\draw[color=gray] (e10.north) -- (v7.south);
			\draw[color=gray] (e11.north) -- (v6.south);
			\draw[color=gray] (e11.north) -- (v10.south);
			\draw[color=gray] (e12.north) -- (v7.south);
			\draw[color=gray] (e12.north) -- (v8.south);
			\draw[color=gray] (e13.north) -- (v7.south);
			\draw[color=gray] (e13.north) -- (v10.south);
			\draw[color=gray] (e14.north) -- (v8.south);
			\draw[color=gray] (e14.north) -- (v9.south);
			\draw[color=gray] (e15.north) -- (v10.south);
			\draw[color=gray] (e15.north) -- (v9.south);
		\end{tikzpicture}
	}
	\caption{The pebble tree contraction poset~$\pebbleTreePoset[3][0,1]$.}
	\label{fig:contractionPoset301}
\end{figure}

\begin{figure}[p]
	\centerline{
		\begin{tikzpicture}[yscale=2.7]
			\begin{scope}[xscale=.9]
      			\node (v1) at (-9,4.1) {\pebbleTree{[ [ $\circ$ [ [ \leaf ] [ $\circ$ [ $\bullet$ [ \leaf ] ] ] ] ] ]}};
      			\node (v2) at (-7,4.1) {\pebbleTree{[ [ $\circ$ [ [ \leaf ] [ $\bullet$ [ $\circ$ [ \leaf ] ] ] ] ] ]}};
      			\node (v3) at (-5,4.1) {\pebbleTree{[ [ [ $\circ$ [ \leaf ] ] [ $\bullet$ [ $\circ$ [ \leaf ] ] ] ] ]}};
      			\node (v4) at (-3,4.1) {\pebbleTree{[ [ [ $\circ$ [ \leaf ] ] [ $\circ$ [ $\bullet$ [ \leaf ] ] ] ] ]}};
      			\node (v5) at (-1,4.1) {\pebbleTree{[ [ $\circ$ [ [ $\circ$ [ \leaf ] ] [ $\bullet$ [ \leaf ] ] ] ] ]}};
      			\node (v6) at (1,4.1) {\pebbleTree{[ [ $\circ$ [ [ $\bullet$ [ \leaf ] ] [ $\circ$ [ \leaf ] ] ] ] ]}};
      			\node (v7) at (3,4.1) {\pebbleTree{[ [ [ $\circ$ [ $\bullet$ [ \leaf ] ] ] [ $\circ$ [ \leaf ] ] ] ]}};
      			\node (v8) at (5,4.1) {\pebbleTree{[ [ [ $\bullet$ [ $\circ$ [ \leaf ] ] ] [ $\circ$ [ \leaf ] ] ] ]}};
      			\node (v9) at (7,4.1) {\pebbleTree{[ [ $\circ$ [ [ $\bullet$ [ $\circ$ [ \leaf ] ] ] [ \leaf ] ] ] ]}};
      			\node (v10) at (9,4.1) {\pebbleTree{[ [ $\circ$ [ [ $\circ$ [ $\bullet$ [ \leaf ] ] ] [ \leaf ] ] ] ]}};
			\end{scope}
			\begin{scope}[xscale=1.2]
      			\node (e1) at (-7,3) {\pebbleTree{[ [ $\circ$ [ [ \leaf ] [ ${\circ}{\bullet}$ [ \leaf ] ] ] ] ]}};
      			\node (e2) at (-6,3) {\pebbleTree{[ [ $\circ$ [ \leaf ] [ $\bullet$ [ $\circ$ [ \leaf ] ] ] ] ]}};
      			\node (e3) at (-5,3) {\pebbleTree{[ [ $\circ$ [ \leaf ] [ $\circ$ [ $\bullet$ [ \leaf ] ] ] ] ]}};
      			\node (e4) at (-4,3) {\pebbleTree{[ [ [ $\circ$ [ \leaf ] ] [ ${\circ}{\bullet}$ [ \leaf ] ] ] ]}};
      			\node (e5) at (-3,3) {\pebbleTree{[ [ $\circ$ [ $\circ$ [ \leaf ] [ $\bullet$ [ \leaf ] ] ] ] ]}};
      			\node (e6) at (-2,3) {\pebbleTree{[ [ $\circ$ [ $\circ$ [ \leaf ] ] [ $\bullet$ [ \leaf ] ] ] ]}};
      			\node (e7) at (-1,3) {\pebbleTree{[ [ $\circ$ [ $\bullet$ [ \leaf ] [ $\circ$ [ \leaf ] ] ] ] ]}};
      			\node (e8) at (0,3) {\pebbleTree{[ [ $\bullet$ [ $\circ$ [ \leaf ] ] [ $\circ$ [ \leaf ] ] ] ]}};
      			\node (e9) at (1,3) {\pebbleTree{[ [ $\circ$ [ $\bullet$ [ $\circ$ [ \leaf ] ] [ \leaf ] ] ] ]}};
      			\node (e10) at (2,3) {\pebbleTree{[ [ $\circ$ [ $\bullet$ [ \leaf ] ] [ $\circ$ [ \leaf ] ] ] ]}};
      			\node (e11) at (3,3) {\pebbleTree{[ [ $\circ$ [ $\circ$ [ $\bullet$ [ \leaf ] ] [ \leaf ] ] ] ]}};
      			\node (e12) at (4,3) {\pebbleTree{[ [ [ ${\circ}{\bullet}$ [ \leaf ] ] [ $\circ$ [ \leaf ] ] ] ]}};
      			\node (e13) at (5,3) {\pebbleTree{[ [ $\circ$ [ $\circ$ [ $\bullet$ [ \leaf ] ] ] [ \leaf ] ] ]}};
      			\node (e14) at (6,3) {\pebbleTree{[ [ $\circ$ [ $\bullet$ [ $\circ$ [ \leaf ] ] ] [ \leaf ] ] ]}};
      			\node (e15) at (7,3) {\pebbleTree{[ [ $\circ$ [ [ ${\circ}{\bullet}$ [ \leaf ] ] [ \leaf ] ] ] ]}};
			\end{scope}
			\begin{scope}[xscale=1.8]
      			\node (f1) at (-3,2) {\pebbleTree{[ [ $\circ$ [ \leaf ] [ ${\circ}{\bullet}$ [ \leaf ] ] ] ]}};
      			\node (f2) at (-2,2) {\pebbleTree{[ [ ${\circ}{\circ}$ [ \leaf ] [ $\bullet$ [ \leaf ] ] ] ]}};
				\node (f3) at (-1,2) {\pebbleTree{[ [ ${\circ}{\bullet}$ [ \leaf ] [ $\circ$ [ \leaf ] ] ] ]}};
      			\node (f4) at (0,2) {\pebbleTree{[ [ $\circ$ [ ${\circ}{\bullet}$ [ \leaf ] [ \leaf ] ] ] ]}};
      			\node (f5) at (1,2) {\pebbleTree{[ [ ${\circ}{\bullet}$ [ $\circ$ [ \leaf ] ] [ \leaf ] ] ]}};
      			\node (f6) at (2,2) {\pebbleTree{[ [ ${\circ}{\circ}$ [ $\bullet$ [ \leaf ] ] [ \leaf ] ] ]}};
      			\node (f7) at (3,2) {\pebbleTree{[ [ $\circ$ [ ${\circ}{\bullet}$ [ \leaf ] ] [ \leaf ] ] ]}};
			\end{scope}
			\begin{scope}[xshift=-1]
      			\node (p) at (0,1.2) {\pebbleTree{[ [ ${\circ}{\circ}{\bullet}$ [ \leaf ] [ \leaf ] ] ]}};
			\end{scope}
			\draw[color=gray] (p.north) -- (f1.south);
			\draw[color=gray] (p.north) -- (f2.south);
			\draw[color=gray] (p.north) -- (f3.south);
			\draw[color=gray] (p.north) -- (f4.south);
			\draw[color=gray] (p.north) -- (f5.south);
			\draw[color=gray] (p.north) -- (f6.south);
			\draw[color=gray] (p.north) -- (f7.south);
			\draw[color=gray] (f1.north) -- (e1.south);
			\draw[color=gray] (f1.north) -- (e2.south);
			\draw[color=gray] (f1.north) -- (e3.south);
			\draw[color=gray] (f1.north) -- (e4.south);
			\draw[color=gray] (f2.north) -- (e3.south);
			\draw[color=gray] (f2.north) -- (e5.south);
			\draw[color=gray] (f2.north) -- (e6.south);
			\draw[color=gray] (f3.north) -- (e2.south);
			\draw[color=gray] (f3.north) -- (e7.south);
			\draw[color=gray] (f3.north) -- (e8.south);
			\draw[color=gray] (f3.north) -- (e10.south);
			\draw[color=gray] (f3.north) -- (e12.south);
			\draw[color=gray] (f4.north) -- (e1.south);
			\draw[color=gray] (f4.north) -- (e5.south);
			\draw[color=gray] (f4.north) -- (e7.south);
			\draw[color=gray] (f4.north) -- (e9.south);
			\draw[color=gray] (f4.north) -- (e11.south);
			\draw[color=gray] (f4.north) -- (e15.south);
			\draw[color=gray] (f5.north) -- (e4.south);
			\draw[color=gray] (f5.north) -- (e6.south);
			\draw[color=gray] (f5.north) -- (e8.south);
			\draw[color=gray] (f5.north) -- (e9.south);
			\draw[color=gray] (f5.north) -- (e14.south);
			\draw[color=gray] (f6.north) -- (e10.south);
			\draw[color=gray] (f6.north) -- (e11.south);
			\draw[color=gray] (f6.north) -- (e13.south);
			\draw[color=gray] (f7.north) -- (e12.south);
			\draw[color=gray] (f7.north) -- (e13.south);
			\draw[color=gray] (f7.north) -- (e14.south);
			\draw[color=gray] (f7.north) -- (e15.south);
			\draw[color=gray] (e1.north) -- (v2.south);
			\draw[color=gray] (e1.north) -- (v1.south);
			\draw[color=gray] (e2.north) -- (v2.south);
			\draw[color=gray] (e2.north) -- (v3.south);
			\draw[color=gray] (e3.north) -- (v1.south);
			\draw[color=gray] (e3.north) -- (v4.south);
			\draw[color=gray] (e4.north) -- (v3.south);
			\draw[color=gray] (e4.north) -- (v4.south);
			\draw[color=gray] (e5.north) -- (v1.south);
			\draw[color=gray] (e5.north) -- (v5.south);
			\draw[color=gray] (e6.north) -- (v4.south);
			\draw[color=gray] (e6.north) -- (v5.south);
			\draw[color=gray] (e7.north) -- (v2.south);
			\draw[color=gray] (e7.north) -- (v6.south);
			\draw[color=gray] (e8.north) -- (v3.south);
			\draw[color=gray] (e8.north) -- (v8.south);
			\draw[color=gray] (e9.north) -- (v5.south);
			\draw[color=gray] (e9.north) -- (v9.south);
			\draw[color=gray] (e10.north) -- (v6.south);
			\draw[color=gray] (e10.north) -- (v7.south);
			\draw[color=gray] (e11.north) -- (v6.south);
			\draw[color=gray] (e11.north) -- (v10.south);
			\draw[color=gray] (e12.north) -- (v7.south);
			\draw[color=gray] (e12.north) -- (v8.south);
			\draw[color=gray] (e13.north) -- (v7.south);
			\draw[color=gray] (e13.north) -- (v10.south);
			\draw[color=gray] (e14.north) -- (v8.south);
			\draw[color=gray] (e14.north) -- (v9.south);
			\draw[color=gray] (e15.north) -- (v10.south);
			\draw[color=gray] (e15.north) -- (v9.south);
		\end{tikzpicture}
	}
	\caption{The pebble tree contraction poset~$\pebbleTreePoset[2][1,1]$.}
	\label{fig:contractionPoset211}
\end{figure}

\begin{definition}
For any node~$c$ (not a leaf, nor the root) with parent~$p$ in a pebble tree~$T$, the \defn{contraction} of~$c$ in~$T$ is the pebble tree~$T/c$ obtained by replacing~$c$ by its children in the list of children of~$p$ and adding to~$p$ the pebbles of~$c$.
The \defn{pebble tree contraction poset}~$\pebbleTreePoset$ is the poset of contractions on pebble trees of~$\pebbleTrees$. 
\end{definition}

\begin{example}
In the extreme situations of \cref{exm:pebbleTrees}:
\begin{itemize}
\item the pebble tree contraction poset~$\pebbleTreePoset[1][b,u]$ is the refinement poset on ordered partitions of~$[b]$,
\item the pebble tree contraction poset~$\pebbleTreePoset[\ell][0,0]$ is the contraction poset on Schr\"oder trees with~$\ell$ leaves.
\end{itemize}
The pebble tree contraction posets~$\pebbleTreePoset[3][0,1]$ and~$\pebbleTreePoset[2][1,1]$ are illustrated in \cref{fig:contractionPoset301,fig:contractionPoset211}.
The fact that~$\pebbleTreePoset[3][0,1]$ and~$\pebbleTreePoset[2][1,1]$ are isomorphic can be seen applying successively Points~\eqref{prop:pebbleTreePosetUprootingMap}, \eqref{prop:pebbleTreePosetRerootingMap}, and~\eqref{prop:pebbleTreePosetBalancingMap} of \cref{prop:pebbleTreePosetOperations} below.
\end{example}

\begin{remark}
\label{rem:rankedContractionPoset}
Observe that:
\begin{itemize}
\item The set of pebble trees~$\pebbleTrees$ is clearly closed under contraction. Hence, the pebble tree contraction poset is a simplicial poset (a poset where each interval is a boolean algebra).
\item The pebble tree contraction poset is ranked: the rank of a pebble tree is \mbox{its number of nodes}.
\item The \defn{maximal} pebble trees of~$\pebbleTreePoset$ are the pebble trees with only unary nodes containing a single pebble and binary nodes containing no pebble (hence, they have $\ell(b+u)-u$ unary nodes and $\ell-1$ binary nodes, thus rank~$\ell (1 + b + u) - u - 1$).
\item The \defn{minimal} pebble tree of~$\pebbleTreePoset$ is the corolla with $\ell$ leaves and $\ell(b+u)-u$ pebbles at the root (hence it has rank~$1$).
\end{itemize}
\end{remark}

Observe now that the mirroring, balancing, inserting, rerooting and uprooting maps of \cref{def:mirroringMap,def:balancingMap,def:insertingMap,def:rerootingMap,def:uprootingMap} obviously commute with contractions.
This implies the following statement.

\begin{proposition}
\label{prop:pebbleTreePosetOperations}
Consider the operations of \cref{def:mirroringMap,def:balancingMap,def:insertingMap,def:rerootingMap,def:uprootingMap}.
\begin{enumerate}
\item \label{prop:pebbleTreePosetMirroringMap} The mirroring map of \cref{def:mirroringMap} defines a poset automorphism of~$\pebbleTreePoset[\ell][b,u]$.
\item \label{prop:pebbleTreePosetBalancingMap} If~$u > 1$, the balancing map of \cref{def:balancingMap} defines a poset isomorphism from the pebble tree contraction poset~$\pebbleTreePoset$ to a principal upper set of the pebble tree contraction poset~$\pebbleTreePoset[\ell][b+1,u-1]$. Hence, $\pebbleTreePoset$ is isomorphic to a principal upper set of~$\pebbleTreePoset[\ell][b+u,0]$.
\item \label{prop:pebbleTreePosetInsertingMap} The inserting map of \cref{def:insertingMap} defines a poset isomorphism from the pebble tree contraction poset~$\pebbleTreePoset$ to a principal upper set of the pebble tree contraction poset~$\pebbleTreePoset[\ell][b+1,u]$.
\item \label{prop:pebbleTreePosetRerootingMap} The rerooting maps of \cref{def:rerootingMap} define poset automorphisms of~$\pebbleTreePoset[\ell][b,0]$.
\item \label{prop:pebbleTreePosetUprootingMap} If~$\ell > 1$, the uprooting map of \cref{def:uprootingMap} defines a poset isomorphism from the pebble tree contraction poset~$\pebbleTreePoset[\ell][0,u]$ to a principal upper set of the pebble tree contraction~poset~$\pebbleTreePoset[\ell-1][u+1,0]$.
\end{enumerate}
\end{proposition}

Finally, we connect the $\alpha$-trees of~\cite{PoirierTradler} to an upper set of the pebble tree contraction poset.

\begin{remark}
\label{rem:bijectionPebbleTreesOrientedTrees2}
Following \cref{rem:bijectionPebbleTreesOrientedTrees1}, observe that for any signature~$\alpha \in \textsc{o} \cdot \{\textsc{i}, \textsc{o}\}^\ell$, the $\alpha$-tree contraction poset is isomorphic to the principal upper set of the pebble tree contraction poset~$\pebbleTreePoset[\ell][1,0]$ generated by the pebble tree whose root has~$\ell - |\alpha|_{\textsc{i}}$ pebbles and $\ell$ children, and whose $i$th children is a leaf if~$\alpha_i = \textsc{o}$ and a unary node with one pebble and a leaf if~$\alpha_i = \textsc{i}$.
\end{remark}

\begin{remark}
The following observations are consequences of \cref{rem:bijectionPebbleTreesOrientedTrees1,rem:bijectionPebbleTreesOrientedTrees2,prop:pebbleTreePosetOperations}:
\begin{itemize}
\item the $\textsc{o}^{\ell+1}$-tree contraction poset is isomorphic to the pebble tree contraction poset~$\pebbleTreePoset[\ell][1,0]$,
\item for any~$\alpha \in \textsc{o} \cdot \{\textsc{i}, \textsc{o}\}^\ell$ with a single occurrence of~$\textsc{i}$, the $\alpha$-tree contraction poset is isomorphic to the pebble tree contraction poset~$\pebbleTreePoset[\ell][0,1]$,
\item for~$\alpha = \textsc{oo} \textsc{i}^{\ell-1}$ or~$\alpha = \textsc{o} \textsc{i}^{\ell-1} \textsc{o}$, the $\alpha$-tree contraction poset is isomorphic to the pebble tree contraction poset~$\pebbleTreePoset[\ell+1][0,0]$ (\ie the contraction poset on Schr\"oder trees).
\end{itemize}
\end{remark}


\subsection{Pebble tree complex}
\label{subsec:pebbleTreeComplex}

As mentioned in \cref{rem:rankedContractionPoset}, the pebble tree contraction poset~$\pebbleTreePoset$ is a simplicial poset.
We now construct the corresponding simplicial complex.
Recall that we denote by~$\leaves$ and~$\balanced$ the sets of leaves and of balanced colors in a subtree~$S$.
We will moreover need the following notations.

\begin{notation}
\label{def:encoding}
For an interval~$L \eqdef [s,t] \subseteq [\ell]$ and a subset~$B \subseteq [b+u]$, we define the sets
\[
L \otimes B \eqdef \bigcup_{p \in B} [\ell p + s - 1, \ell p + t - 1]
\qquad\text{and}\qquad
L \boxtimes B \eqdef [s,t-1] \cup ( L \otimes B ).
\]
\end{notation}

\begin{definition}
\label{def:pebbleTreeComplex}
The \defn{pebble tree complex}~$\pebbleTreeComplex$ is the simplicial complex whose simplices are the sets~$\Lambda(T) \eqdef \set{\lambda(S)}{S \text{ subtree of } T}$ for all pebble trees~$T \in \pebbleTrees$, where~$\lambda(S) \eqdef \leaves \boxtimes \balanced$.
\end{definition}

\begin{example}
In the extreme situations of \cref{exm:pebbleTrees}:
\begin{itemize}
\item the simplices of the pebble tree complex~$\pebbleTreeComplex[1][b,u]$ are the flags~$B_1 \subsetneq B_2 \subsetneq \dots \subsetneq B_k \subseteq [b]$,
\item the simplices of the pebble tree complex~$\pebbleTreePoset[\ell][0,0]$ are the collections of pairwise nested or non-adjacents intervals of~$[\ell-1]$.
\end{itemize}
\cref{fig:pebbleTreeSimplices} illustrates some more generic examples of simplices~$\Lambda(T)$.

\begin{figure}
	\centerline{
		\pebbleTree{[ [ {}, label={[label distance=0pt]180:{$12345$}} [ {$\circ$}, label={[label distance=-4pt]180:{$134$}} [ {}, label={[label distance=0pt]180:{$1$}} [ {$\circ$}, label={[label distance=-4pt]180:{$3$}} [ \leaf ] ] [ {$\bullet$}, label={[label distance=-4pt]0:{$\mathrlap{7}$}} [ \leaf ] ] ] ] [ {$\circ$}, label={[label distance=-4pt]0:{$58$}} [ {$\bullet$}, label={[label distance=-4pt]0:{$8$}} [ \leaf ] ] ] ] ]}[7pt][6pt]
		\quad
		\pebbleTree{[ [{$\circ$}, label={[label distance=-4pt]180:{$12345$}} [{$\circ$}, label={[label distance=-4pt]180:{$12$}} [{}, label={[label distance=0pt]180:{$1$}},scale=.01 [{$\circ$}, label={[label distance=-4pt]180:{$3$}} [ \leaf ] ] [{$\bullet$}, label={[label distance=-4pt]0:{$7$}} [ \leaf ] ] ] [{$\bullet$}, label={[label distance=-4pt]0:{$8$}} [ \leaf ] ] ] ] ]}
		\quad
		\pebbleTree{[ [{${\circ}{\circ}$}, label={[label distance=-4pt]180:{$12345$}} [{}, label={[label distance=0pt]180:{$1$}} [{$\circ$}, label={[label distance=-4pt]180:{$3$}} [ \leaf ] ] [{$\bullet$}, label={[label distance=-4pt]0:{$\mathrlap{7}$}} [ \leaf ] ] ] [{$\bullet$}, label={[label distance=-4pt]0:{$8$}} [ \leaf ] ] ] ]}
		\quad
		\pebbleTree{[ [{}, label={[label distance=1pt]180:{$12345$}} [{$\bullet$}, label={[label distance=0pt]180:{$13467$}} [{$\circ$}, label={[label distance=-4pt]180:{$3$}} [ \leaf ] ] [{$\bullet$}, label={[label distance=-4pt]0:{$\mathrlap{7}$}} [{$\circ$}, label={[label distance=-4pt]0:{$4$}} [ \leaf ] ] ] ] [{$\circ$}, label={[label distance=-4pt]0:{$5$}} [ \leaf ] ] ] ]}
		\quad
		\pebbleTree{[ [{${\circ}{\circ}{\circ}{\bullet}{\bullet}$}, label={[label distance=1pt]180:{$12345$}} [ \leaf ] [ \leaf ] [ \leaf ] ] ]}
	}
	\caption{Some $\circ$-balanced and $\bullet$-unbalanced $\{\circ, \bullet\}$-pebble trees~$T$ and the associated simplices~$\Lambda(T)$. Each node~$n$ of~$T$ is labeled by the concatenation of the elements of the set~$\lambda(T_n)$.}
	\label{fig:pebbleTreeSimplices}
\end{figure}
\end{example}

\begin{proposition}
\label{prop:pebbleTreeComplex}
The pebble tree complex~$\pebbleTreeComplex$ is a pseudomanifold, whose face poset is isomorphic to the pebble tree poset~$\pebbleTreePoset$.
\end{proposition}

\begin{proof}
Observe first that~$\Lambda(T/n) = \Lambda(T) \ssm \{\lambda(T_n)\}$ for any node~$n$ (not a leaf, nor the root) in a pebble tree~$T$.
Hence, the face poset of~$\pebbleTreeComplex$ is indeed isomorphic to~$\pebbleTreePoset$.
We thus obtain that~$\pebbleTreeComplex$ is a pure simplicial complex since~$\pebbleTreePoset$ is a ranked simplicial poset.
It remains to prove that~$\pebbleTreeComplex$ is a pseudomanifold, meaning that any ridge (\ie codimension~$1$ face) is contained in precisely two facets (\ie maximal dimensional faces).
Consider thus a pebble tree of corank~$1$, and let~$p$ be the only node which is neither unary with a pebble, nor binary with no pebble.
We distinguish two cases:
\begin{itemize}
\item If~$p$ has three children and no pebble, then there are two ways to open~$p$ as usual (see \cref{fig:pseudomanifold}\,(left)) and it does not matter whether~$p$ and its children are balanced or not for each pebble color.
\item If~$p$ has two children and a $\pebbleColor$-pebble, then $p$ is $\pebbleColor$-balanced, and there are still two ways to open~$p$ depending on whether its children are $\pebbleColor$-balanced or not (see \cref{fig:pseudomanifold}\,(middle)).
\item If~$p$ has one child and two pebbles of different colors, then there are still two ways to open~$p$ choosing which pebble goes in the parent and which pebble goes in the child (see \cref{fig:pseudomanifold}\,(right)).
\qedhere
\end{itemize}
\begin{figure}
	\centerline{
		\begin{tikzpicture}[xscale=.8, yscale=2]
			\node (R) at (-1,1) {\pebbleTree{[ [ [ [ $X$ ] [ $Y$ ] ] [ $Z$ ] ] ]}[9pt]};
			\node (S) at (0,0) {\pebbleTree{[ [ [ $X$ ] [ $Y$ ] [ $Z$ ] ] ]}[9pt]};
			\node (T) at (1,1) {\pebbleTree{[ [ [ $X$ ] [ [ $Y$ ] [ $Z$ ] ] ] ]}[9pt]};
			\draw[color=gray] (S.north) -- (R.south);
			\draw[color=gray] (S.north) -- (T.south);
		\end{tikzpicture}
		\quad
		\begin{tikzpicture}[xscale=.8, yscale=2]
			\node (R) at (-1,1) {\pebbleTree{[ [ [ $\bullet$ [ $X$ ] ] [ $\mathrlap{\phantom{Y}}\smash{Y_\bullet}$ ] ] ]}};
			\node (S) at (0,0) {\pebbleTree{[ [ $\bullet$ [ $X$ ] [ $\mathrlap{\phantom{Y}}\smash{Y_\bullet}$ ] ] ]}};
			\node (T) at (1,1) {\pebbleTree{[ [ $\bullet$ [ [ $X$ ] [ $\mathrlap{\phantom{Y}}\smash{Y_\bullet}$ ] ] ] ]}};
			\draw[color=gray] (S.north) -- (R.south);
			\draw[color=gray] (S.north) -- (T.south);
		\end{tikzpicture}
		\quad
		\begin{tikzpicture}[xscale=.8, yscale=2]
			\node (R) at (-1,1) {\pebbleTree{[ [ [ $\mathrlap{\phantom{X}}\smash{X_\bullet}$ ] [ $\bullet$ [ $Y$ ] ] ] ]}};
			\node (S) at (0,0) {\pebbleTree{[ [ $\bullet$ [ $\mathrlap{\phantom{X}}\smash{X_\bullet}$ ] [ $Y$ ] ] ]}};
			\node (T) at (1,1) {\pebbleTree{[ [ $\bullet$ [ [ $\mathrlap{\phantom{X}}\smash{X_\bullet}$ ] [ $Y$ ] ] ] ]}};
			\draw[color=gray] (S.north) -- (R.south);
			\draw[color=gray] (S.north) -- (T.south);
		\end{tikzpicture}
		\quad
		\begin{tikzpicture}[xscale=.8, yscale=2]
			\node (R) at (-1,1) {\pebbleTree{[ [ [ $\bullet$ [ $X$ ] ] [ $Y$ ] ] ]}};
			\node (S) at (0,0) {\pebbleTree{[ [ $\bullet$ [ $X$ ] [ $Y$ ] ] ]}};
			\node (T) at (1,1) {\pebbleTree{[ [ [ $X$ ] [ $\bullet$ [ $Y$ ] ] ] ]}};
			\draw[color=gray] (S.north) -- (R.south);
			\draw[color=gray] (S.north) -- (T.south);
		\end{tikzpicture}
		\quad
		\begin{tikzpicture}[xscale=.8, yscale=2]
			\node (R) at (-1,1) {\pebbleTree{[ [ $\circ$ [ $\bullet$ [ $X$ ] ] ] ]}};
			\node (S) at (0,0) {\pebbleTree{[ [ $\circ{\bullet}$ [ $X$ ] ] ]}};
			\node (T) at (1,1) {\pebbleTree{[ [ $\bullet$ [ $\circ$ [ $X$ ] ] ] ]}};
			\draw[color=gray] (S.north) -- (R.south);
			\draw[color=gray] (S.north) -- (T.south);
		\end{tikzpicture}
	}
	\caption{The pebble tree complex is a pseudomanifold. All possible corank $1$ pebble trees are obtained by contracting precisely two maximal pebble trees. In the second (resp.~third) picture, we mark~$Y$ (resp.~$X$) with a $\bullet$ to indicate that it is $\bullet$-balanced. Neither~$X$ nor~$Y$ are $\bullet$-balanced in the fourth picture.}
	\label{fig:pseudomanifold}
\end{figure}
\end{proof}

\begin{remark}
In contrast to the special situations of \cref{exm:pebbleTrees}, the pebble tree complex is not flag in general.
For instance, $\pebbleTreeComplex[2][1,0]$ and $\pebbleTreeComplex[3][0,1]$ are not flag.
\end{remark}

Finally, we translate \cref{prop:pebbleTreePosetOperations,rem:bijectionPebbleTreesOrientedTrees2} to the pebble tree complex. 

\begin{proposition}
\label{prop:pebbleTreeComplexOperations}
Consider the operations of \cref{def:mirroringMap,def:balancingMap,def:insertingMap,def:rerootingMap,def:uprootingMap}.
\begin{enumerate}
\item \label{prop:pebbleTreeComplexMirroringMap} The map defined by $\ell j + i - \delta_{j \ne 0} \mapsto \ell (j + 1) - i$ for any~$(i,j) \in  ([\ell] \times [0,b+u]) \ssm \{(\ell,0)\}$ induces an automorphism of the pebble tree complex~$\pebbleTreeComplex[\ell][b,u]$.
\item \label{prop:pebbleTreeComplexBalancingMap} If~$u > 1$, the pebble tree complex~$\pebbleTreeComplex$ is isomorphic to the link of the face~$[\ell] \boxtimes [b]$ in the pebble tree complex~$\pebbleTreeComplex[\ell][b+1,u-1]$. Hence, $\pebbleTreeComplex$ is isomorphic to a link of~$\pebbleTreeComplex[\ell][b+u,0]$.
\item \label{prop:pebbleTreeComplexInsertingMap} The pebble tree complex~$\pebbleTreeComplex$ is isomorphic to the link of the face~$\set{\{i\} \boxtimes [1]}{i \in [\ell]}$ in the pebble tree complex~$\pebbleTreeComplex[\ell][b+1,u]$.
\item \label{prop:pebbleTreeComplexRerootingMap} The rerooting maps of \cref{def:rerootingMap} induce automorphisms of the pebble tree complex~$\pebbleTreeComplex[\ell][b,0]$.
\item \label{prop:pebbleTreeComplexUprootingMap} If~$\ell > 1$, the pebble tree complex~$\pebbleTreeComplex[\ell][0,u]$ is isomorphic to the link of the face~$\set{\{i\} \boxtimes [1]}{i \in [2,\ell-1]}$ in the pebble tree complex~$\pebbleTreeComplex[\ell-1][u+1,0]$.
\end{enumerate}
\end{proposition}

\begin{remark}
\label{rem:bijectionPebbleTreesOrientedTrees3}
Following \cref{rem:bijectionPebbleTreesOrientedTrees1,rem:bijectionPebbleTreesOrientedTrees2}, observe that for any signature~$\alpha \in \{\textsc{i}, \textsc{o}\}^{\ell+1}$, the $\alpha$-tree complex is isomorphic to the link of the face~$\set{\{i\} \boxtimes [1]}{\alpha_i = \textsc{i}}$ in the pebble tree complex~$\pebbleTreeComplex[\ell][1,0]$.
\end{remark}


\subsection{Pebble tree flip graph}
\label{subsec:pebbleTreeFlipGraph}

As the pebble tree complex~$\pebbleTreeComplex$ is a pseudomanifold by \cref{prop:pebbleTreeComplex}, it is natural to consider its dual graph.

\begin{definition}
Two maximal pebble trees~$T$ and~$T'$ of~$\pebbleTreePoset$ are related by a \defn{flip} if there are nodes~$n$ of~$T$ and~$n'$ of~$T'$ such that the following equivalent conditions hold:
\begin{itemize}
\item the contraction $T/n$ coincides with the contraction~$T'/n'$.
\item $\Lambda(T) \ssm \{\lambda(T_n)\} = \Lambda(T') \ssm \{\lambda(T'_{n'})\}$,
\end{itemize}
All possible types of flips are illustrated in \cref{fig:pebbleTreeFlips}.
The \defn{flip graph} is the graph whose vertices are the maximal pebble trees of~$\pebbleTreePoset$ and whose edges are the flips between them.
\end{definition}

\pagebreak

\begin{example}
In the extreme situations of \cref{exm:pebbleTrees}:
\begin{itemize}
\item the flip graph on~$\pebbleTrees[1][b,u]$ is the graph of adjacent transpositions on \mbox{permutations of~$[b]$},
\item the flip graph on~$\pebbleTrees[\ell][0,0]$ is the rotation graph on binary trees with~$\ell$ leaves.
\end{itemize}
These extreme situations correspond to the right and left cases of \cref{fig:pebbleTreeFlips} respectively.
\cref{fig:exmPebbleTreeFlips} illustrates a sequence of flips in maximal pebble trees of~$\pebbleTrees[3][1,1]$.
\cref{fig:flipGraph301,fig:flipGraph211} illustrate the flip graphs on maximal pebble trees of~$\pebbleTrees[3][0,1]$ and~$\pebbleTrees[2][1,1]$ (which are isomorphic by \cref{prop:pebbleTreePosetOperations} as already mentioned before).

\begin{figure}
	\centerline{
		\pebbleTree{[ [ [ [ $X$ ] [ $Y$ ] ] [ $Z$ ] ] ]}[10pt] \raisebox{-1cm}{$\longleftrightarrow$} \pebbleTree{[ [ [ $X$ ] [ [ $Y$ ] [ $Z$ ] ] ] ]}[10pt]
		\quad
		\pebbleTree{[ [ [ $\bullet$ [ $X$ ] ] [ $\mathrlap{\phantom{Y}}\smash{Y_\bullet}$ ] ] ]} \raisebox{-1cm}{$\longleftrightarrow$} \pebbleTree{[ [ $\bullet$ [ [ $X$ ] [ $\mathrlap{\phantom{Y}}\smash{Y_\bullet}$ ] ] ] ]}[7.5pt]
		\quad 
		\pebbleTree{[ [ [ $\mathrlap{\phantom{X}}\smash{X_\bullet}$ ] [ $\bullet$ [ $Y$ ] ] ] ]} \raisebox{-1cm}{$\longleftrightarrow$} \pebbleTree{[ [ $\bullet$ [ [ $\mathrlap{\phantom{X}}\smash{X_\bullet}$ ] [ $Y$ ] ] ] ]}[7.5pt]
		\quad 
		\pebbleTree{[ [ [ $\bullet$ [ $X$ ] ] [ $Y$ ] ] ]} \raisebox{-1cm}{$\longleftrightarrow$} \pebbleTree{[ [ [ $X$ ] [ $\bullet$ [ $Y$ ] ] ] ]}
		\quad 
		\pebbleTree{[ [ $\circ$ [ $\bullet$ [ $X$ ] ] ] ]}[5.5pt] \raisebox{-1cm}{$\longleftrightarrow$} \pebbleTree{[ [ $\bullet$ [ $\circ$ [ $X$ ] ] ] ]}[5.5pt]
	}
	\caption{All possible flips in a maximal pebble tree. In the second (resp.~third) picture, we mark~$Y$ (resp.~$X$) with a $\bullet$ to indicate that it is $\bullet$-balanced. Neither~$X$ nor~$Y$ are $\bullet$-balanced in the fourth picture.}
	\label{fig:pebbleTreeFlips}
\end{figure}

\begin{figure}[p]
	\centerline{
		\pebbleTree{[ [ [ $\circ$ [ [ $\circ$ [ \leaf ] ] [ $\bullet$ [ \leaf ] ] ] ] [ $\circ$ [ $\bullet$ [ \leaf ] ] ] ] ]}
		\pebbleTree{[ [ $\circ$ [ [ $\circ$ [ [ $\circ$ [ \leaf ] ] [ $\bullet$ [ \leaf ] ] ] ] [ $\bullet$ [ \leaf ] ] ] ] ]}
		\pebbleTree{[ [ $\circ$ [ [ [ $\circ$ [ \leaf ] ] [ $\bullet$ [ \leaf ] ] ] [ $\circ$ [ $\bullet$ [ \leaf ] ] ] ] ] ]}
		\pebbleTree{[ [ $\circ$ [ [ $\circ$ [ \leaf ] ] [ [ $\bullet$ [ \leaf ] ] [ $\circ$ [ $\bullet$ [ \leaf ] ] ] ] ] ] ]}
		\pebbleTree{[ [ $\circ$ [ [ $\circ$ [ \leaf ] ] [ [ $\bullet$ [ \leaf ] ] [ $\bullet$ [ $\circ$ [ \leaf ] ] ] ] ] ] ]}
		\pebbleTree{[ [ $\circ$ [ [ $\circ$ [ \leaf ] ] [ $\bullet$ [ [ $\bullet$ [ \leaf ] ] [ $\circ$ [ \leaf ] ] ] ] ] ] ]}
		\pebbleTree{[ [ $\circ$ [ $\bullet$ [ [ $\circ$ [ \leaf ] ] [ [ $\bullet$ [ \leaf ] ] [ $\circ$ [ \leaf ] ] ] ] ] ] ]}
		\pebbleTree{[ [ $\circ$ [ $\bullet$ [ [ [ $\circ$ [ \leaf ] ] [ $\bullet$ [ \leaf ] ] ] [ $\circ$ [ \leaf ] ] ] ] ] ]}
		\pebbleTree{[ [ $\bullet$ [ $\circ$ [ [ [ $\circ$ [ \leaf ] ] [ $\bullet$ [ \leaf ] ] ] [ $\circ$ [ \leaf ] ] ] ] ] ]}
	}
	\caption{A sequence of flips in maximal $\circ$-balanced and $\bullet$-unbalanced $\{\circ, \bullet\}$-pebble trees.}
	\label{fig:exmPebbleTreeFlips}
\end{figure}

\begin{figure}[p]
	\centerline{
		\begin{tikzpicture}[xscale=1, yscale=.7]
			\node (a) at (0,-4) {\pebbleTree{[ [ [ $\bullet$ [ \leaf ] ] [ [ $\bullet$ [ \leaf ] ] [ \leaf ] ] ] ]}};
			\node (b) at (0,4) {\pebbleTree{[ [ [ [ \leaf ] [ $\bullet$ [ \leaf ] ] ] [ $\bullet$ [ \leaf ] ] ] ]}};
			\node (c) at (-8,-2) {\pebbleTree{[ [ [ \leaf ] [ $\bullet$ [ [ $\bullet$ [ \leaf ] ] [ \leaf ] ] ] ] ]}};
			\node (d) at (8,2) {\pebbleTree{[ [ [ $\bullet$ [ [ \leaf ] [ $\bullet$ [ \leaf ] ] ] ] [ \leaf ] ] ]}};
			\node (e) at (-8,2) {\pebbleTree{[ [ [ \leaf ] [ [ $\bullet$ [ \leaf ] ] [ $\bullet$ [ \leaf ] ] ] ] ]}};
			\node (f) at (8,-2) {\pebbleTree{[ [ [ [ $\bullet$ [ \leaf ] ] [ $\bullet$ [ \leaf ] ] ] [ \leaf ] ] ]}};
			\node (g) at (-5,0) {\pebbleTree{[ [ [ \leaf ] [ $\bullet$ [ [ \leaf ] [ $\bullet$ [ \leaf ] ] ] ] ] ]}};
			\node (h) at (5,0) {\pebbleTree{[ [ [ $\bullet$ [ [ $\bullet$ [ \leaf ] ] [ \leaf ] ] ] [ \leaf ] ] ]}};
			\node (i) at (-2,-1) {\pebbleTree{[ [ [ $\bullet$ [ \leaf ] ] [ [ \leaf ] [ $\bullet$ [ \leaf ] ] ] ] ]}};
			\node (j) at (2,1) {\pebbleTree{[ [ [ [ $\bullet$ [ \leaf ] ] [ \leaf ] ] [ $\bullet$ [ \leaf ] ] ] ]}};
			\draw[<->, color=gray] (a) -- (c);
			\draw[<->, color=gray] (b) -- (d);
			\draw[<->, color=gray] (c) -- (e);
			\draw[<->, color=gray] (d) -- (f);
			\draw[<->, color=gray] (e) -- (b);
			\draw[<->, color=gray] (f) -- (a);
			\draw[<->, color=gray] (c) -- (g);
			\draw[<->, color=gray] (d) -- (h);
			\draw[<->, color=gray] (e) -- (g);
			\draw[<->, color=gray] (f) -- (h);
			\draw[<->, color=gray] (g) -- (i);
			\draw[<->, color=gray] (h) -- (j);
			\draw[<->, color=gray] (a) -- (i);
			\draw[<->, color=gray] (b) -- (j);
			\draw[<->, color=gray] (i) -- (j);
		\end{tikzpicture}
	}
	\caption{The flip graph on pebble trees of~$\pebbleTrees[3][0,1]$.}
	\label{fig:flipGraph301}
\end{figure}

\begin{figure}[p]
	\centerline{
		\begin{tikzpicture}[xscale=1, yscale=.7]
			\node (a) at (0,-4) {\pebbleTree{[ [ $\circ$ [ [ \leaf ] [ $\bullet$ [ $\circ$ [ \leaf ] ] ] ] ] ]}};
			\node (b) at (0,4) {\pebbleTree{[ [ $\circ$ [ [ $\bullet$ [ $\circ$ [ \leaf ] ] ] [ \leaf ] ] ] ]}};
			\node (c) at (-8,-2) {\pebbleTree{[ [ $\circ$ [ [ \leaf ] [ $\circ$ [ $\bullet$ [ \leaf ] ] ] ] ] ]}};
			\node (d) at (8,2) {\pebbleTree{[ [ $\circ$ [ [ $\circ$ [ $\bullet$ [ \leaf ] ] ] [ \leaf ] ] ] ]}};
			\node (e) at (-8,2) {\pebbleTree{[ [ $\circ$ [ [ $\circ$ [ \leaf ] ] [ $\bullet$ [ \leaf ] ] ] ] ]}};
			\node (f) at (8,-2) {\pebbleTree{[ [ $\circ$ [ [ $\bullet$ [ \leaf ] ] [ $\circ$ [ \leaf ] ] ] ] ]}};
			\node (g) at (-5,0) {\pebbleTree{[ [ [ $\circ$ [ \leaf ] ] [ $\circ$ [ $\bullet$ [ \leaf ] ] ] ] ]}};
			\node (h) at (5,0) {\pebbleTree{[ [ [ $\circ$ [ $\bullet$ [ \leaf ] ] ] [ $\circ$ [ \leaf ] ] ] ]}};
			\node (i) at (-2,-1) {\pebbleTree{[ [ [ $\circ$ [ \leaf ] ] [ $\bullet$ [ $\circ$ [ \leaf ] ] ] ] ]}};
			\node (j) at (2,1) {\pebbleTree{[ [ [ $\bullet$ [ $\circ$ [ \leaf ] ] ] [ $\circ$ [ \leaf ] ] ] ]}};
			\draw[<->, color=gray] (a) -- (c);
			\draw[<->, color=gray] (b) -- (d);
			\draw[<->, color=gray] (c) -- (e);
			\draw[<->, color=gray] (d) -- (f);
			\draw[<->, color=gray] (e) -- (b);
			\draw[<->, color=gray] (f) -- (a);
			\draw[<->, color=gray] (c) -- (g);
			\draw[<->, color=gray] (d) -- (h);
			\draw[<->, color=gray] (e) -- (g);
			\draw[<->, color=gray] (f) -- (h);
			\draw[<->, color=gray] (g) -- (i);
			\draw[<->, color=gray] (h) -- (j);
			\draw[<->, color=gray] (a) -- (i);
			\draw[<->, color=gray] (b) -- (j);
			\draw[<->, color=gray] (i) -- (j);
		\end{tikzpicture}
	}
	\caption{The flip graph on pebble trees of~$\pebbleTrees[2][1,1]$.}
	\label{fig:flipGraph211}
\end{figure}
\end{example}

As the dual graph of a pure simplicial pseudomanifold, the pebble tree flip graph is regular.
Its degree is~$\ell (1 + b + u) - u - 2$.
As we will see in \cref{thm:pebbleTreePolytope1} that it is the graph of a simple polytope, it has the connectivity of its degree.
Among various further properties of this graph that would require more investigations, we mention the following problem in connection to~\cite{SleatorTarjanThurston, Pournin}.

\begin{problem}
Evaluate the diameter of the flip graph on maximal pebble trees of~$\pebbleTrees$.
\end{problem}

Finally, note that \cref{prop:pebbleTreePosetOperations,rem:bijectionPebbleTreesOrientedTrees2} directly translate to morphisms between the flip graphs on the corresponding trees.


\section{Pebble tree geometry}
\label{sec:pebbleTreeGeometry}

\enlargethispage{.5cm}
This section is devoted to the geometry of pebble trees.
After quickly reminding some geometric preliminaries (\cref{subsec:geometricPreliminaries}), we construct the pebble tree fan (\cref{subsec:pebbleTreeFan}) and the pebble tree polytope (\cref{subsec:pebbleTreePolytope}).


\subsection{Geometric preliminaries}
\label{subsec:geometricPreliminaries}

We refer to \cite{Ziegler-polytopes} for a reference on polyhedral geometry, and only remind the basic notions needed later in the paper.

A (polyhedral) \defn{cone} is the positive span~$\R_{\ge0}\rays[]$ of a finite set~$\rays[]$ of vectors of~$\R^d$ or equivalently, the intersection of finitely many closed linear half-spaces of~$\R^d.$ 
The \defn{faces} of a cone are its intersections with its supporting hyperplanes. 
The \defn{rays} (resp.~\defn{facets}) are the faces of dimension~$1$ (resp.~ codimension~$1$).
A cone is \defn{simplicial} if its rays are linearly independent.
A (polyhedral) \defn{fan}~$\Fan$ is a set of cones such that any face of a cone of~$\Fan$ belongs to~$\Fan$, and any two cones of~$\Fan$ intersect along a face of both. 
A fan is \defn{essential} if the intersection of its cones is the origin, \defn{complete} if the union of its cones covers~$\R^d$, and \defn{simplicial} if all its cones are simplicial.

Note that a simplicial fan defines a simplicial complex on its rays (the simplices of the simplicial complex are the subsets of rays which span a cone of the fan).
Conversely, given a simplicial complex~$\Delta$ with ground set~$V$, one can try to realize it geometrically by associating a ray~$\ray[v]$ of~$\R^d$ to each~$v \in V$, and the cone~$\R_{\ge0}\rays[\triangle]$ generated by the set~$\rays[\triangle] \eqdef \set{\ray[v]}{v \in \triangle}$ to each~$\triangle \in \Delta$.
To show that the resulting cones indeed form a fan, we will need the following statement, which can be seen as a reformulation of~\cite[Coro.~4.5.20]{DeLoeraRambauSantos}.

\begin{proposition}
\label{prop:characterizationFan}
Consider a closed simplicial pseudomanifold~$\Delta$ with ground set~$V$ and a set of vectors~$(\ray[v])_{v \in V}$ of~$\R^d$, and define~$\rays[\triangle] \eqdef \set{\ray[v]}{v \in \triangle}$ for any~$\triangle \in \Delta$.
Then the collection of cones~$\set{\R_{\ge 0}\rays[\triangle]}{\triangle \in \Delta}$ forms a complete simplicial fan of~$\R^d$ if and~only~if
\begin{itemize}
\item there exists a vector~$\b{v}$ of~$\R^d$ contained in only one of the open cones~$\R_{> 0}\rays[\triangle]$ for~$\triangle \in \Delta$,
\item for any two adjacent facets~$\triangle, \triangle'$ of~$\Delta$ with~$\triangle \ssm \{v\} = \triangle' \ssm \{v'\}$, we have~$\alpha_v \alpha_{v'} > 0$~where
\[
\alpha_v \, \ray[v] + \alpha_{v'} \, \ray[v'] + \sum_{w \in \triangle \cap \triangle'} \alpha_w \, \ray[w] = 0
\]
denotes the unique (up to rescaling) linear dependence on~$\rays[\triangle \cup \triangle']$.
\end{itemize}
\end{proposition}

A \defn{polytope} is the convex hull of finitely many points of~$\R^d$ or equivalently, a bounded intersection of finitely many closed affine half-spaces of~$\R^d$.
The \defn{faces} of a polytope are its intersections with its supporting hyperplanes.
The \defn{vertices} (resp.~\defn{edges}, resp.~\defn{facets}) are the faces of dimension~$0$ (resp.~dimension~$1$, resp.~codimension~$1$).

The \defn{normal cone} of a face~$\polytope{F}$ of a polytope~$\polytope{P}$ is the cone generated by the normal vectors to the supporting hyperplanes of~$\polytope{P}$ containing~$\polytope{F}$.
Said differently, it is the cone of vectors~$\b{c}$ of~$\R^d$ such that the linear form~$\b{x} \mapsto \dotprod{\b{c}}{\b{x}}$ on~$\polytope{P}$ is maximized by all points of the face~$\polytope{F}$.
The \defn{normal fan} of~$\polytope{P}$ is the set of normal cones of all its faces.

Consider now a complete simplicial fan~$\Fan$ of~$\R^d$ with rays~$(\ray[v])_{v \in V}$ and cones~$\R_{\ge0} \rays[\triangle]$ for~${\triangle \in \Delta}$, where~$\rays[\triangle] \eqdef \set{\ray[v]}{v \in \triangle}$ as in \cref{prop:characterizationFan}.
To realize the fan~$\Fan$, one can try to pick a height vector~$\b{h} \eqdef (h_v)_{v \in V} \in \R^V$ and consider the polytope
\(
\polytope{P}_\b{h} \eqdef \set{\b{x} \in \R^d}{\dotprod{\ray[v]}{\b{x}} \le h_v \text{ for all } v \in V}.
\)
The following classical statement characterizes the height vectors~$\b{h}$ for which the fan~$\Fan$ is the normal fan of this polytope~$\polytope{P}_\b{h}$.
We borrow the formulation from~\cite[Lem.~2.1]{ChapotonFominZelevinsky}.

\begin{proposition}
\label{prop:characterizationPolytopalFan}
Let~$\Fan$ be an essential complete simplicial fan in~$\R^n$ with rays~$(\ray[v])_{v \in V}$ and cones~$\R_{\ge0} \rays[\triangle]$ for~$\triangle \in \Delta$.
Then the following are equivalent for any height vector~$\b{h} \in \R^V$:
\begin{itemize}
\item The fan~$\Fan$ is the normal fan of the polytope~$\polytope{P}_\b{h} \eqdef \set{\b{x} \in \R^d}{\dotprod{\ray[v]}{\b{x}} \le h_v \text{ for all } v \in V}$.
\item For two adjacent facets~$\triangle, \triangle'$ of~$\Delta$ with~$\triangle \ssm \{v\} = \triangle' \ssm \{v'\}$, the height vector~$\b{h}$ satisfies the \defn{wall crossing inequality}
\[
\alpha_v \, h_v + \alpha_{v'} \, h_{v'} + \sum_{w \in \triangle \cap \triangle'} \alpha_w \, h_w > 0
\]
where
\[
\alpha_v \, \ray[v] + \alpha{v'} \, \ray[v'] + \sum_{w \in \triangle \cap \triangle'} \alpha_w \, \ray[w] = 0
\]
denotes the unique linear dependence on~$\rays[\triangle \cup \triangle']$ such that~$\alpha_v + \alpha_{v'} = 2$.
\end{itemize}
\end{proposition}


\subsection{Pebble tree fan}
\label{subsec:pebbleTreeFan}

Fix~$\ell, b, u \in \N$ and consider the intervals
\[
I_0 \eqdef [\ell (b+1)-1]
\qquad\text{and}\qquad
I_i \eqdef [\ell (b+i), \ell (b+i+1) - 1]
\quad\text{for all } i \in [u]
\]
whose union is the interval
\[
I \eqdef I_0 \sqcup I_1 \sqcup \dots \sqcup I_u = [\ell(b+u+1)-1].
\]
We work in the Euclidean space~$\R^I$ with canonical basis~$(\b{e}_i)_{i \in I}$.
We denote by~$\smash{\one_J \eqdef \sum_{j \in J} \b{e}_j}$ the characteristic vector of a subset~$J \subseteq I$.
As our constructions actually live in the linear subspace
\[
\HH_\ell^{b,u} \eqdef \set{\b{x} \in \R^I}{\dotprod{\one_{I_i}}{\b{x}} = 0 \text{ for all } 0 \le i \le u},
\]
we define the vector
\[
\ray[J] \eqdef \sum_{i = 0}^u \big( |I_i \ssm J| \cdot \one_{I_i \cap J} - |I_i \cap J| \cdot \one_{I_i \ssm J} \big) \in \HH_\ell^{b,u}
\]
for each subset~$J \subseteq I$.
It is immediate to check that these vectors satisfy the linear dependences
\[
\ray[J] + \ray[K] = \ray[J \cup K] + \ray[J \cap K]
\]
for any~$J, K \subseteq I$.
Finally, we associate to any pebble subtree~$S$ the vector~$\ray[S] \eqdef \ray[\lambda(S)] = \ray[\leaves \boxtimes \balanced]$ where~$\leaves$ and~$\balanced$ denote the sets of leaves and of balanced colors in~$S$, and the operation~$\boxtimes$ was defined in \cref{def:encoding}.
Note that~$\ray[S] = \b{0}$ when~$S$ is the entire tree~$T$ (because~$\leaves[T] = [\ell]$ and~$\balanced[T] = [b]$ so that~$\lambda(T) = I_0$) or when~$S$ is a leaf~$i$ (because~$\leaves[S] = \{i\}$ and~$\balanced[S] = \varnothing$ so that~$\lambda(S) = \varnothing$).
We now use these vectors~$\ray[S]$ to construct the pebble tree fan.

\begin{definition}
\label{def:pebbleTreeFan}
The \defn{pebble tree fan}~$\pebbleTreeFan$ is the collection of cones~$\Cone(T) \eqdef \cone \set{\ray[S]}{S \text{ subtree of } T}$ for all pebble trees~$T \in \pebbleTrees$, where~$\ray[S] \eqdef \ray[\lambda(S)] = \ray[\leaves \boxtimes \balanced]$.
\end{definition}

\pagebreak

\begin{example}
\label{exm:braidFanSylvesterFan}
In the extreme situations of \cref{exm:pebbleTrees}:
\begin{itemize}
\item the pebble tree fan~$\pebbleTreeFan[1][b,u]$ is the braid fan, with a ray~$\ray[J]$ for each proper subset~${\varnothing \neq J \subsetneq [b]}$ and a maximal cone~$\Cone(\sigma)$ for each permutation~$\sigma$ of~$[b]$, defined by the inequalities ${x_{\sigma(1)} \le \dots \le x_{\sigma(b)}}$,
\item the pebble tree fan~$\pebbleTreeFan[\ell][0,0]$  is the sylvester fan, with a ray~$\ray[J]$ for each proper interval~$J$ of~$[\ell]$ and a maximal cone $\Cone(T)$ for each binary tree~$T$, defined by the inequalities~$x_i \le x_j$ whenever there is a path from~$i$ to~$j$ in the tree~$T$ labeled in inorder and oriented towards its root. 
\end{itemize}
Note that the sylvester fan coarsens the braid fan: the cone~$\Cone(T)$ of the sylvester fan can also be obtained by glueing the cones~$\Cone(\sigma)$ of the braid fan corresponding to the linear extensions~$\sigma$ of~$T$.
\end{example}

\begin{theorem}
\label{thm:pebbleTreeFan}
The pebble tree fan~$\pebbleTreeFan$ is an essential complete simplicial fan in~$\HH_\ell^{b,u}$, whose face lattice is the pebble tree contraction poset~$\pebbleTreePoset$.
\end{theorem}

The proof of \cref{thm:pebbleTreeFan} relies on the description of the linear dependences among adjacent maximal cones described in \cref{lem:linearDependence}.
To obtain these dependences, we need the following preliminary statement, where we use the operation~$\otimes$ defined in \cref{def:encoding}.

\begin{lemma}
\label{lem:linearCombinations}
For any maximal pebble tree~$S$ and any~$B \subseteq \balanced$, there are in~$S$ some distinct unary subtrees~$U_1, \dots, U_k$ with children~$V_1, \dots, V_k$ respectively such that~${\ray[{\leaves[S] \otimes B}] = \sum_{i \in [k]} \ray[U_i] - \ray[V_i]}$.
\end{lemma}

\begin{proof}
If a subtree~$U$ has a $\pebbleColor$-pebble and a unique child~$V$, then we have~$\ray[{\leaves[U] \otimes \{\pebbleColor\}}] = \ray[U] - \ray[V]$ because~$\leaves[U] = \leaves[V]$ and~$\balanced[U] = \balanced[V] \sqcup \{\pebbleColor\}$.
Hence, for any~$\pebbleColor \in \balanced$, if we denote by~$U_1, \dots, U_k$ the closest descendants of~$S$ with a $\pebbleColor$-pebble and by~$V_1, \dots, V_k$ their respective children, then we have~$\ray[{\leaves[S] \otimes \{\pebbleColor\}}] =  \sum_{i \in [k]} \ray[{\leaves[U_i] \otimes \{\pebbleColor\}}]  = \sum_{i \in [k]} \ray[U_i] - \ray[V_i]$ because~$\leaves[S] = \bigsqcup_{i \in [k]} \leaves[U_i]$.
The result follows since~$\ray[{\leaves[S] \otimes B}] = \sum_{\pebbleColor \in B} \ray[{\leaves[S] \otimes \{\pebbleColor\}}]$.
\end{proof}

\begin{lemma}
\label{lem:linearDependence}
Let~$T$ and~$T'$ be two adjacent maximal pebble trees and let~$S$ and~$S'$ be the subtrees~of~$T$ and~$T'$ such that~$\Lambda(T) \ssm \{\lambda(S)\} = \Lambda(T') \ssm \{\lambda(S')\}$.
Then there is a linear dependence among the rays~$\ray[R]$ associated to the subtrees~$R$ of~$T$ and~$T'$ where the rays~$\ray[S]$ and~$\ray[S']$ both have coefficient~$1$.
\end{lemma}

\begin{proof}
We analyse the five possible types of flips described in \cref{fig:pebbleTreeFlips}.
In all cases, we denote by~$R$ the parent of~$S$ and~$S'$.

\bigskip
\parpic(3cm,1cm)(-5pt, 5pt)[r][b]{
	\mbox{\pebbleTree{[ [ {}, label={[label distance=-2pt]135:{$R$}} [ {}, label={[label distance=-4pt]135:{$S$}} [ $X$ ] [ $Y$ ] ] [ $Z$ ] ] ]}[9pt] \!\!\raisebox{-1cm}{$\longleftrightarrow$}\!\! \pebbleTree{[ [ {}, label={[label distance=-2pt]45:{$R$}} [ $X$ ] [ {}, label={[label distance=-4pt]45:{$S'$}} [ $Y$ ] [ $Z$ ] ] ] ]}[9pt]}
}{
\noindent
\textbf{Case 1.}
We have \\[.1cm]
$\begin{array}{lll}
{\leaves[S] = \leaves[X] \sqcup \leaves[Y]} & {\leaves[S'] = \leaves[Y] \sqcup \leaves[Z]} & {\leaves[R] = \leaves[X] \sqcup \leaves[Y] \sqcup \leaves[Z]} \\
{\balanced[S] = \balanced[X] \cap \balanced[Y]} & {\balanced[S'] = \balanced[Y] \cap \balanced[Z]} & {\balanced[R] = \balanced[X] \cap \balanced[Y] \cap \balanced[Z]}
\end{array}$\\[.1cm]
which yields
}

\[\boxed{
\ray[S] + \ray[S'] = \ray[R] + \ray[Y] + \ray[{\leaves[S] \otimes (\balanced[S] \ssm \balanced[R])}] + \ray[{\leaves[S'] \otimes (\balanced[S'] \ssm \balanced[R])}] - \ray[{\leaves[Y] \otimes \balanced[Y]}]
}\]

\medskip\noindent
Since~$\ray[{\leaves[S] \otimes (\balanced[S] \ssm \balanced[R])}]$ (resp.~$\ray[{\leaves[S'] \otimes (\balanced[S'] \ssm \balanced[R])}]$, resp.~$\ray[{\leaves[Y] \otimes \balanced[Y]}]$) is a linear combination of the rays~$\ray[P]$ for some subtrees~$P$ of~$S$ distinct from~$S$ (resp.~of~$S'$ distinct from~$S'$, resp.~of~$Y$) by \cref{lem:linearCombinations}, this is indeed a linear dependence among the rays~$\ray[Q]$ associated to the subtrees~$Q$ of~$T$ and~$T'$ where the rays~$\ray[S]$ and~$\ray[S']$ both have coefficient~$1$.

\bigskip
\parpic(3cm,1cm)(-5pt, 5pt)[r][b]{
	\mbox{\pebbleTree{[ [ {}, label={[label distance=-2pt]135:{$R$}} [ {$\bullet$}, label={[label distance=-4pt]180:{$S$}} [ $X$ ] ] [ $\mathrlap{\phantom{Y}}\smash{Y_\bullet}$ ] ] ]} \!\!\raisebox{-1cm}{$\longleftrightarrow$}\!\! \pebbleTree{[ [ {$\bullet$}, label={[label distance=-4pt]180:{$R$}} [ {}, label={[label distance=-3pt]45:{$S'$}} [ $X$ ] [ $\mathrlap{\phantom{Y}}\smash{Y_\bullet}$ ] ] ] ]}[7.5pt]}
}{
\noindent
\textbf{Case 2.}
We have \\[.1cm]
$\begin{array}{lll}
{\leaves[S] = \leaves[X]} & {\leaves[S'] = \leaves[X] \sqcup \leaves[Y]} & {\leaves[R] = \leaves[X] \sqcup \leaves[Y]} \\
{\balanced[S] = \balanced[X] \sqcup \{\bullet\}} & {\balanced[S'] = \balanced[X] \cap \balanced[Y]} & {\balanced[R] = ( \balanced[X] \cap \balanced[Y] ) \sqcup \{\bullet\}}
\end{array}$\\[.1cm]
which yields
\(\qquad\boxed{
\ray[S] + \ray[S'] = \ray[R] + \ray[X] - \ray[{\leaves[Y] \otimes \{\bullet\}}]
}\)
}

\medskip\noindent
Again, we can develop $\ray[{\leaves[Y] \otimes \{\bullet\}}]$ using \cref{lem:linearCombinations}, so that we indeed obtained a linear dependence among the rays~$\ray[Q]$ for the subtrees~$Q$ of~$T$ and~$T'$ where the rays~$\ray[S]$ and~$\ray[S']$ both have coefficient~$1$.

\bigskip
\parpic(3cm,1cm)(-5pt, -12pt)[r][b]{
	\mbox{\pebbleTree{[ [ {}, label={[label distance=-2pt]135:{$R$}} [ $\mathrlap{\phantom{X}}\smash{X_\bullet}$ ] [ {$\bullet$}, label={[label distance=-4pt]0:{$S$}} [ $Y$ ] ] ] ]} \!\!\raisebox{-1cm}{$\longleftrightarrow$}\!\! \pebbleTree{[ [ {$\bullet$}, label={[label distance=-4pt]180:{$R$}} [ {}, label={[label distance=-3pt]45:{$S'$}} [ $\mathrlap{\phantom{X}}\smash{X_\bullet}$ ] [ $Y$ ] ] ] ]}[7.5pt]}
}{
\noindent
\textbf{Case 3.} Symmetric to Case~2.
}

\bigskip
\parpic(3cm,1cm)(-5pt, 12pt)[r][b]{
	\mbox{\pebbleTree{[ [ {}, label={[label distance=-2pt]135:{$R$}} [ {$\bullet$}, label={[label distance=-4pt]180:{$S$}} [ $X$ ] ] [ $Y$ ] ] ]} \!\!\raisebox{-1cm}{$\longleftrightarrow$}\!\! \pebbleTree{[ [ {}, label={[label distance=-2pt]135:{$R$}} [ $X$ ] [ {$\bullet$}, label={[label distance=-4pt]0:{$S'$}} [ $Y$ ] ] ] ]}}
}{
\noindent
\textbf{Case 4.}
Assume first that $R$ has an ancestor with a $\bullet$-pebble.
Then we additionally denote
\begin{itemize}
\item by~$U_0$ the closest ancestor of~$R$ which has a~$\bullet$-pebble,
\item by $U_1, \dots, U_k$ the closest descendants of~$U_0$ but not descendants of~$R$ \\ which have a~$\bullet$-pebble,
\item by~$V_0, V_1, \dots, V_k$ the (unique) children of~$U_0, U_1, \dots, U_k$ respectively.
\end{itemize}
}

\noindent
We have
\[
\begin{array}{lll}
{\leaves[S] = \leaves[X]} & {\leaves[S'] = \leaves[Y]} & {\leaves[U_i] = \leaves[V_i]} \\
{\balanced[S] = \balanced[X] \sqcup \{\bullet\}} & {\balanced[S'] = \balanced[Y] \sqcup \{\bullet\}} &  {\balanced[U_i] = \balanced[V_i] \sqcup \{\bullet\}}
\end{array}
\]
and moreover $\leaves[U_0] = \leaves[S] \sqcup \leaves[S'] \sqcup \leaves[U_1] \sqcup \dots \sqcup \leaves[U_k]$.
Using \cref{lem:linearCombinations}, we get
\[
\ray[U_0] - \ray[V_0] = \ray[{\leaves[U_0] \otimes \{\bullet\}}] = \ray[S] - \ray[X] + \ray[S'] - \ray[Y] + \ray[U_1] - \ray[V_1] + \dots + \ray[U_k] - \ray[V_k]
\]
or, written differently
\[
\boxed{
\ray[S] + \ray[S'] = \ray[X] + \ray[Y] + \ray[U_0] - \ray[V_0] - \ray[U_1] + \ray[V_1] + \dots - \ray[U_k] + \ray[V_k]
}\]

\medskip
Now if~$R$ has no ancestor with a $\bullet$-pebble, then using that~$\ray[{\leaves[T] \otimes \{\bullet\}}] = \b{0}$, we obtain similarly
\[
\boxed{
\ray[S] + \ray[S'] = \ray[X] + \ray[Y] - \ray[U_1] + \ray[V_1] + \dots - \ray[U_k] + \ray[V_k]
}\]
where
\begin{itemize}
\item $U_1, \dots, U_k$ are the closest descendants of the root of~$T$ but not descendants of~$R$ which have a $\bullet$-pebble,
\item $V_1, \dots, V_k$ are the (unique) children of~$U_1, \dots, U_k$ respectively.
\end{itemize}

\bigskip
\parpic(3cm,1cm)(-5pt, 5pt)[r][b]{
	\mbox{\pebbleTree{[ [ {$\circ$}, label={[label distance=-4pt]180:{$R$}} [ {$\bullet$}, label={[label distance=-4pt]180:{$S$}} [ $X$ ] ] ] ]}[5.5pt] \!\!\raisebox{-1cm}{$\longleftrightarrow$}\!\! \pebbleTree{[ [ {$\bullet$}, label={[label distance=-4pt]0:{$R$}} [ {$\circ$}, label={[label distance=-4pt]0:{$S'$}} [ $X$ ] ] ] ]}[5.5pt]}
}{
\noindent
\textbf{Case 5.}
We have\\[.1cm]
$\begin{array}{lll}
{\leaves[S] = \leaves[X]} & {\leaves[S'] = \leaves[X]} & {\leaves[R] = \leaves[X]} \\
{\balanced[S] = \balanced[X] \sqcup \{\bullet\}} & {\balanced[S'] = \balanced[X] \sqcup \{\circ\}} & {\balanced[R] = \balanced[X] \sqcup \{\circ, \bullet\}}
\end{array}$\\[.1cm]
which yields
\(\qquad\boxed{
\ray[S] + \ray[S'] = \ray[R] + \ray[X]
}\)
}

\end{proof}

\begin{proof}[Proof of \cref{thm:pebbleTreeFan}]
Note that $\pebbleTreeFan$ is included in~$\HH_\ell^{b,u}$ since all rays~$\ray$ are.
To prove that it is a complete simplicial fan, we just check the two criteria of \cref{prop:characterizationFan}.
The second criterion is guaranteed by the description of the linear dependences in \cref{lem:linearDependence}.
For the first criterion, consider the vector \[\b{v} = \sum_{i \in [\ell-2]} \ray[{[i]}] + \sum_{i \in [b]} 2^{\ell + i} \ray[{[\ell i, \ell(i+1)-1]}] + \sum_{i \in [u]} 2^{\ell + b+i} \ray[{[\ell i + 1, \ell(i+1)-1]}],\] and a pebble tree~$T$ such that~$\b{v}$ is contained in the interior of~$\Cone(T)$.
As the last~$\ell-1$ coordinates of~$\b{v}$ are strictly larger than all other coordinates, each of the last~$\ell-1$ leaves of~$T$ is preceded by a unary node with pebble colored by~$b+u$. Repeating the argument, we obtain that the first leaf of~$T$ is preceded by a chain of unary nodes with pebbles colored~$1, \dots, b$ while each of the last~$\ell-1$ leaves if~$T$ is preceded by a chain of unary nodes with pebbles colored~$1, \dots, b+u$. Finally, we obtain that the rest of the tree~$T$ is the left comb since it is the only Schr\"oder tree whose cone in the sylvester fan contains the vector~$\sum_{i \in [\ell-2]} \ray[{[i]}]$.
Finally, $\pebbleTreeFan$ is essential as the dimension of its cones matches the dimension of~$\HH_\ell^{b,u}$.
\end{proof}

\begin{remark}
A few observations on the pebble tree fan:
\begin{itemize}
\item The simple descriptions of \cref{exm:braidFanSylvesterFan} for the braid fan and for the sylvester fan unfortunately fail for arbitrary~$b,u \ge 0$. Indeed, there is a natural way to label the nodes of a maximal pebble tree: label the binary nodes in inorder by~${[\ell-1]}$ and the unary nodes by the only leaf first covered by this pebble. This labeling yields a cone $\Cone(T)$ for each maximal pebble tree~$T$, defined by~$x_i \le x_j$ whenever there is a directed path from~$i$ to~$j$ in the tree~$T$ oriented towards its root. However, the cones $\Cone(T)$ for all maximal pebble trees~$T$ do not define a complete simplicial fan (check out the case~$\ell = 2$, $b = 1$ and~$u = 0$). In fact, our pebble tree fan~$\pebbleTreeFan$ is not refined by the braid fan in general
\item Our definition of the pebble tree fan~$\pebbleTreeFan$ respects some symmetries of the pebble tree complex~$\pebbleTreeComplex$ but not all. See \cref{prop:pebbleTreePolytopeOperations} for a precise statement directly on polytopes.
\item \cref{lem:linearDependence} actually proves that the pebble tree fan~$\pebbleTreeFan$ is smooth, meaning that the principal vectors spanning the rays of any maximal cone of~$\pebbleTreeFan$ form an integral basis of the space (in other words, the corresponding toric variety is smooth).
\end{itemize}
\end{remark}


\subsection{Pebble tree polytope}
\label{subsec:pebbleTreePolytope}

Our next step is to construct a polytope whose normal fan is the pebble tree fan, using the criterion of \cref{prop:characterizationPolytopalFan}.

\begin{definition}
\label{def:submodularFunction}
A \defn{submodular function} on~$n$ is a map~$f$ from the subsets of~$[n]$ to~$\R_{\ge0}$ such that~$f_\varnothing = 0$ and
\[
f_{A \cup B} + f_{A \cap B} \le f_A + f_B
\]
for any subsets~$A$ and~$B$ of~$[n]$.
We then define
\[
\Delta f \eqdef \min \big( f_A + f_B - f_{A \cup B} - f_{A \cap B} \big)
\]
where the minimum ranges over all subsets~$A$ and~$B$ of~$[n]$ such that~$A \not\subseteq B$ and~$A \not\supseteq B$.
Note~that
\begin{itemize}
\item $\sum_{i \in [k]} f_{A_i} - f_A \ge (k-1) \cdot \Delta f$ for any~$A = \bigcup_{i \in [k]} A_i$ where~$A_1, \dots, A_k$ are pairwise disjoint, 
\item $\Delta \lambda f = \lambda \Delta f$ for any scalar factor~$\lambda$.
\end{itemize}
We say that~$f$ is \defn{strictly submodular} when~$\Delta f > 0$.
\end{definition}

\begin{theorem}
\label{thm:pebbleTreePolytope1}
Pick three strictly submodular functions~$f$ on~$\ell$, $g$ on~$\ell$, and $h$ on~$b+u$ such that
\[
\Delta f > 4 (\ell b + \ell u - u) \cdot (\max g + \max h)
\qquad\text{and}\qquad
\Delta g > (\ell b + \ell u - u + 1) \cdot \max h.
\]
Then the pebble tree fan~$\pebbleTreeFan$ is the normal fan of the \defn{pebble tree polytope}~$\pebbleTreePolytope(f,g,h)$, the $(\ell + \ell b + \ell u - u - 2)$-dimensional polytope defined in the subspace~$\HH_\ell^{b,u}$ by the inequalities
\[
\dotprod{\ray[{[s,t] \boxtimes B}]}{\b{x}} \le f_{[s,t]} + g_{[s,t]} \cdot |B| + h_B
\]
for all~$1 \le s \le t \le \ell$ and all~$B \subseteq [b+u]$.
\end{theorem}

\begin{proof}
To shorten notations in this proof, we define for a maximal pebble subtree~$S$
\[
f_S \eqdef f_{\leaves[S]},
\qquad
g_S \eqdef g_{\leaves[S]},
\qquad
h_S \eqdef h_{\balanced[S]},
\qquad\text{and}\qquad
\varphi_S \eqdef f_{\leaves[S]} + g_{\leaves[S]} \cdot |\balanced[S]| + h_{\balanced[S]}.
\]
Observe that if~$V$ is the unique child of~$U$, then~$\varphi_U - \varphi_V = g_U + h_U - h_V$ because~$\leaves[U] = \leaves[V]$ and~$|\balanced[U]| = |\balanced[V]| + 1$.
We just need to prove that the function~$\varphi$ satisfies the wall-crossing inequalities of \cref{prop:characterizationPolytopalFan} for each of the linear dependences boxed in the proof of \cref{lem:linearDependence}.

\medskip\noindent{\bf Case 1.}
By \cref{lem:linearCombinations,lem:linearDependence}, we have
\[
\ray[S] + \ray[S'] - \ray[R] - \ray[Y] - \sum_{i \in [k]} ( \ray[U_i] - \ray[V_i] ) - \sum_{i \in [k']} ( \ray[U'_i] - \ray[V'_i] ) + \sum_{i \in [k'']} ( \ray[U''_i] - \ray[V''_i] ) = \b{0}
\]
for distinct unary subtrees~$U_1, \dots, U_k$ of~$S$ (resp.~$U'_1, \dots, U'_{k'}$ of~$S'$, resp.~$U''_1, \dots, U''_{k''}$ of~$Y$) with respective children~$V_1, \dots, V_k$ (resp.~$V'_1, \dots, V'_{k'}$, resp.~$V''_1, \dots, V''_{k''}$).
Since~$\leaves[R] = \leaves[S] \cup \leaves[S']$ and~$\leaves[Y] = \leaves[S] \cap \leaves[S']$, we have
\begin{align*}
f_S + f_{S'} & - f_R - f_Y \ge \Delta f > 4 (\ell b + \ell u - u) (\max g + \max h) \\
& \ge ( g_R + h_R ) + ( g_Y + h_Y ) + \sum_{i \in [k]} ( g_{U_i} + h_{U_i} ) + \sum_{i \in [k']} ( g_{U'_i} + h_{U'_i} ) + \sum_{i \in [k'']} ( g_{V''_i} + h_{V''_i} ),
\end{align*}
where the last inequality holds since~$U_i \ne U_j$ (resp.~$U'_i \ne U'_j$, resp.~$V''_i \ne V''_j$) for~$i \ne j$, and the pebble tree~$T$ has~$\ell b + \ell u - u$ unary subtrees.
Since~$f, g, h$ take non-negative values and~$\varphi_U - \varphi_V = g_U + h_U - h_V$ when~$V$ is the unique child of~$U$, we obtain that~$\varphi$ satisfies the wall-crossing inequality
\[
\varphi_S + \varphi_{S'} - \varphi_R - \varphi_Y - \sum_{i \in [k]} ( \varphi_{U_i} - \varphi_{V_i} ) - \sum_{i \in [k']} ( \varphi_{U'_i} - \varphi_{V'_i} ) + \sum_{i \in [k'']} ( \varphi_{U''_i} - \varphi_{V''_i} ) > 0.
\]

\medskip\noindent{\bf Case 2.}
By \cref{lem:linearDependence}, we have
\[
( \ray[R] - \ray[S'] ) - ( \ray[S] - \ray[X] ) - \sum_{i \in [k]} ( \ray[U_i] - \ray[V_i] ) = \b{0}
\]
for distinct unary subtrees~$U_1, \dots, U_k$ of~$Y$ with children~$V_1, \dots, V_k$, such that~${\bigsqcup_{i \in [k]} \leaves[U_i] = \leaves[Y]}$.
Since~$\leaves[R] = \leaves[S] \sqcup \leaves[Y]$, we obtain that
\begin{align*}
g_R - g_S - \sum_{i \in [k]} g_{U_i} \ge (k-1) \cdot \Delta g \ge \Delta g > (\ell b + \ell u - u + 1) \cdot \max h \ge h_{S'} + h_S + \sum_{i \in [k]} h_{U_i},
\end{align*}
where the last inequality holds since~$S \ne U_i \ne U_j$ for~$i \ne j$, and the subtree~$R$ has at most~$\ell b + \ell u - u$ unary subtrees.
Since~$f, g, h$ take non-negative values and~${\varphi_U - \varphi_V = g_U + h_U - h_V}$ when~$V$ is the unique child of~$U$, we obtain that~$\varphi$ satisfies the wall-crossing inequality
\[
( \varphi_R - \varphi_{S'} ) - ( \varphi_S - \varphi_X ) - \sum_{i \in [k]} ( \varphi_{U_i} - \varphi_{V_i} ) > 0.
\]

\medskip\noindent{\bf Case 3.}
Symmetric to Case~2.

\bigskip\noindent{\bf Case 4.}
Assume first that~$R$ has an ancestor with a $\bullet$-pebble.
Then by \cref{lem:linearDependence}, we have
\[
( \ray[U_0] - \ray[V_0] ) - ( \ray[S] - \ray[X] ) - ( \ray[S'] - \ray[Y] ) - \sum_{i \in [k]} ( \ray[U_i] - \ray[V_i] ) = \b{0}.
\]
Since~$\leaves[U_0] = \leaves[S] \sqcup \leaves[S'] \sqcup \bigsqcup_{i \in [k]} \leaves[U_i]$, we obtain that
\[
g_{U_0} - g_S - g_{S'} - \sum_{i \in [k]} g_{U_i} \ge (k+1) \cdot \Delta g \ge \Delta g > (\ell b + \ell u - u + 1) \cdot \max h > h_{V_0} + h_S + h_{S'} + \sum_{i \in [k]} h_{U_i},
\]
where the last inequality holds since~$S \ne U_i \ne U_j \ne S'$ for~$i \ne j$, and the subtree~$U_0$ has at most~$\ell b + \ell u - u$ unary subtrees.
Since~$f, g, h$ take non-negative values and~${\varphi_U - \varphi_V = g_U + h_U - h_V}$ when~$V$ is the unique child of~$U$, we obtain that~$\varphi$ satisfies the wall-crossing inequality
\[
(  \varphi_{U_0} -  \varphi_{V_0} ) - (  \varphi_S -  \varphi_X ) - (  \varphi_{S'} -  \varphi_Y ) - \sum_{i \in [k]} (  \varphi_{U_i} -  \varphi_{V_i} ) > 0.
\]

Assume now that~$R$ has no ancestor with a $\bullet$-pebble.
Then we have
\[
( \ray[S] - \ray[X] ) + ( \ray[S'] - \ray[Y] ) + \sum_{i \in [k]} ( \ray[U_i] - \ray[V_i] ) = \b{0}.
\]
The wall-crossing inequality is thus even easier to satisfy since~$\varphi_{U_0} - \varphi_{V_0}$ does not appear.

\bigskip\noindent{\bf Case 5.}
By \cref{lem:linearDependence}, we have
\[
\ray[S] + \ray[S'] - \ray[R] - \ray[X] = \b{0}
\]
Since~$\leaves[R] = \leaves[S] = \leaves[S'] = \leaves[X]$ and~$\balanced[R] = \balanced[S] \sqcup \{\circ\} = \balanced[S'] \sqcup \{\bullet\} = \balanced[X] \sqcup \{\circ, \bullet\}$, we have
\[
\varphi_S + \varphi_{S'} - \varphi_R - \varphi_X = h_S + h_{S'} - h_R - h_R > 0.
\qedhere
\]
\end{proof}

\begin{remark}
Note that the conditions of \cref{thm:pebbleTreePolytope1} are just sufficient conditions to ensure the wall-crossing inequalities.
To find functions satisfying these conditions, pick three arbitrary strictly submodular functions~$f, g, h$ and rescale first~$g$ by a factor~$4 (\ell b + \ell u - u + 1) \cdot \max h / \Delta g$, and then~$f$ by a factor~$(\ell b + \ell u - u) \cdot (\max g + \max h) / \Delta f$.
We just write~$\pebbleTreePolytope$ if we want to consider~$\pebbleTreePolytope(f,g,h)$ for arbitrary~$f,g,h$ satisfying the conditions of \cref{thm:pebbleTreePolytope1}.
\end{remark}

\pagebreak

\begin{example}
In the extreme situations of \cref{exm:pebbleTrees}:
\begin{itemize}
\item the pebble tree fan~$\pebbleTreeFan[1][b,u]$ (\aka braid fan) is the normal fan of the classical permutahedron, which can be obtained for $h_B = \binom{b+1}{2} - \binom{|B|+1}{2}$ (the functions $f$ and~$g$ are irrelevant here, since $[s,t]$ is constant to~$[1]$),
\item the pebble tree fan~$\pebbleTreeFan[\ell][0,0]$ (\aka sylvester fan) is the normal fan of the classical associahedron~\cite{ShniderSternberg,Loday}, which can be obtained for~$f_X = \binom{\ell+1}{2} - \binom{|X|+1}{2}$ (the functions~$g$ and~$h$ are irrelevant here, since~$B$ is constant to~$\varnothing$).
\end{itemize}
\cref{fig:pebbleTreePolytopes} illustrates polytopal realizations of the pebble tree fans~$\pebbleTreeFan[3][0,1]$ and~$\pebbleTreeFan[2][1,1]$.
Note that, while they have the same combinatorics by~\cref{prop:pebbleTreePosetOperations}, their geometric realizations differ.
\begin{figure}
	\centerline{\includegraphics[scale=.3]{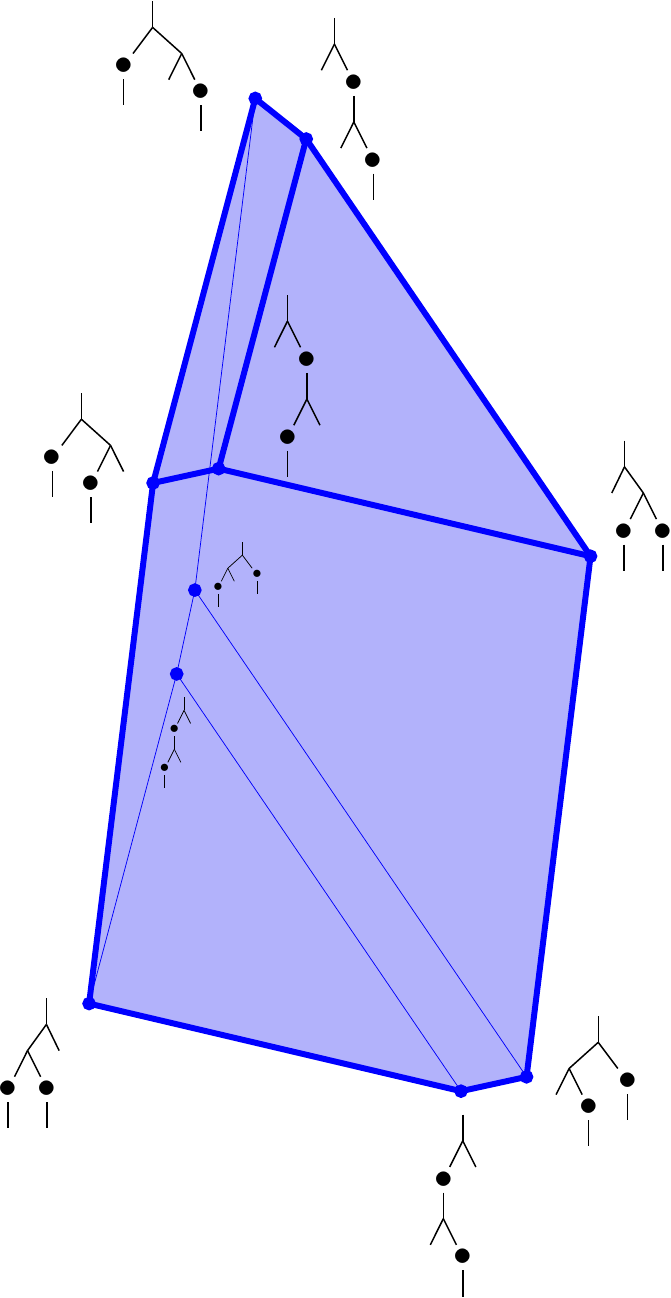} \hspace{.3cm} \includegraphics[scale=.3]{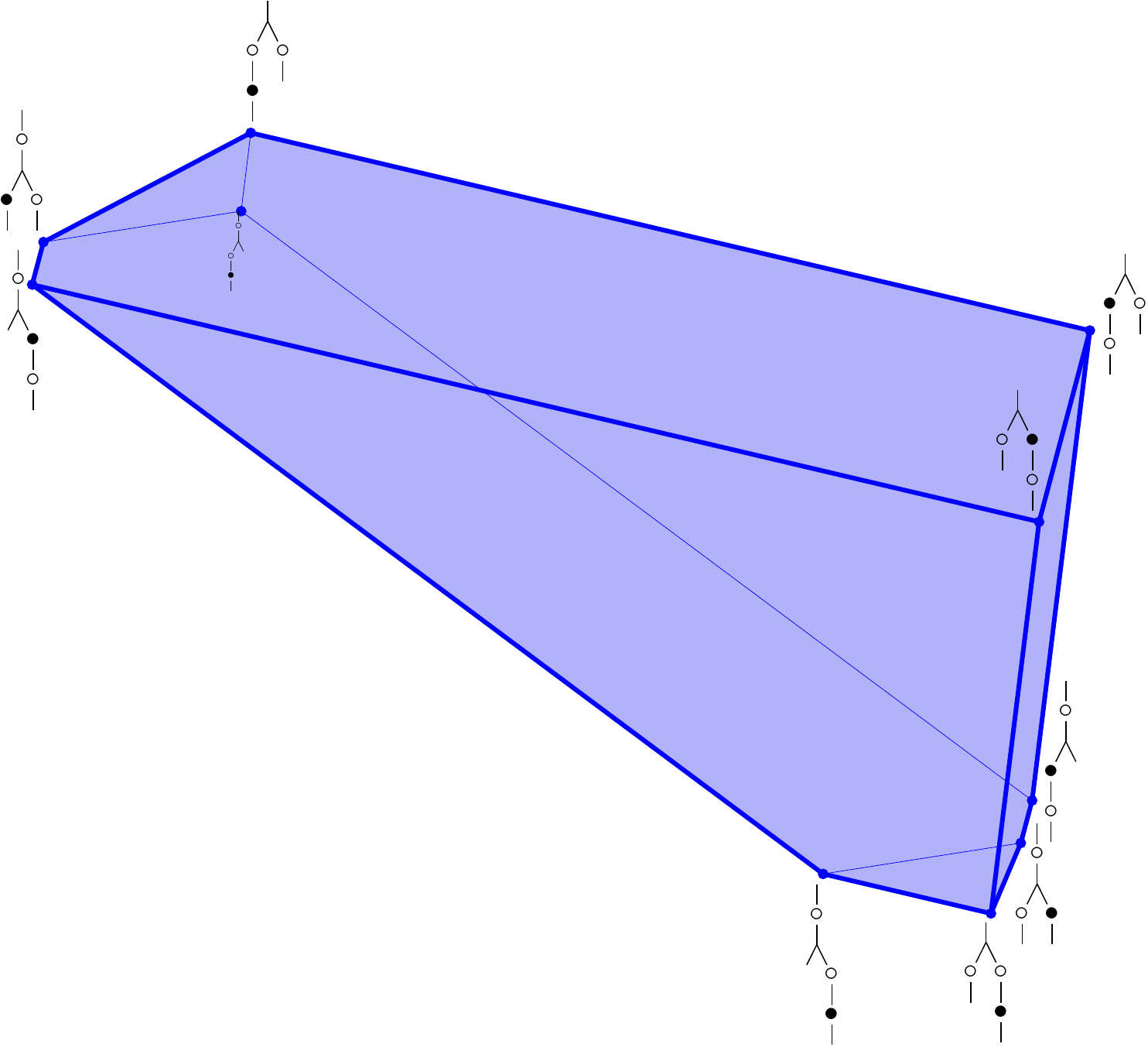}}
	\caption{The pebble tree polytopes~$\pebbleTreePolytope[3][0,1]$ (left) and~$\pebbleTreePolytope[2][1,1]$ (right).}
	\label{fig:pebbleTreePolytopes}
\end{figure}
\end{example}

Finally, we translate the first three points of \cref{prop:pebbleTreePosetOperations,rem:bijectionPebbleTreesOrientedTrees2} to pebble tree polytopes.
Note that the last two transformations of \cref{prop:pebbleTreePosetOperations} do not respect the geometry of the pebble tree polytopes.

\begin{proposition}
\label{prop:pebbleTreePolytopeOperations}
Consider the operations of \cref{def:mirroringMap,def:balancingMap,def:insertingMap}.
\begin{enumerate}
\item \label{prop:pebbleTreePolytopeMirroringMap} The map defined by $\b{e}_{\ell j + i - \delta_{j \ne 0}} \mapsto \b{e}_{\ell (j + 1) - i}$ for any~$(i,j) \in  ([\ell] \times [0,b+u]) \ssm \{(\ell,0)\}$ induces an isometry of the pebble tree polytope~$\pebbleTreePolytope[\ell][b,u](f,g,h)$.
\item \label{prop:pebbleTreePolytopeBalancingMap}
If~$u > 1$, the pebble tree polytope~$\pebbleTreePolytope[\ell][b,u](f,g,h)$ is a facet of the pebble tree polytope $\pebbleTreePolytope[\ell][b+1,u-1](f,g,h)$. Hence, $\pebbleTreePolytope[\ell][b,u](f,g,h)$ is a codimension~$u$ face of $\pebbleTreePolytope[\ell][b+u,0](f,g,h)$.
\item \label{prop:pebbleTreePolytopeInsertingMap} The pebble tree polytope~$\pebbleTreePolytope[\ell][b,u](f,g,h)$ is a codimension~$\ell$ face of the pebble tree polytope~$\pebbleTreePolytope[\ell][b+1,u](f,g,h')$ where~$h(X) = h'(\set{x+1}{x \in X})$ for~$X \subseteq [b+u]$.
\end{enumerate}
\end{proposition}

\begin{remark}
\label{rem:bijectionPebbleTreesOrientedTrees4}
Following \cref{rem:bijectionPebbleTreesOrientedTrees1,rem:bijectionPebbleTreesOrientedTrees2,rem:bijectionPebbleTreesOrientedTrees3}, observe that for any signature~$\alpha \in \{\textsc{i}, \textsc{o}\}^{\ell+1}$, the $\alpha$-assocoipahedron of~\cite{PoirierTradler}  is realized by a face of the pebble tree polytope~$\pebbleTreePolytope[\ell][1,0]$.
For instance, \cref{fig:assocoipahedra} shows faces of the pebble tree polytopes~$\pebbleTreePolytope[4][1,0]$ and~$\pebbleTreePolytope[3][1,0]$ which realize the $\alpha$-assocoipahedra for $\alpha = \textsc{oiioi}$ and~$\alpha = \textsc{oooi}$ presented in \cite[Figs.~8, 9, 14 \& 15]{PoirierTradler}.
Note that the combinatorics of the $\textsc{oooi}$-assocoipahedra represented in \cref{fig:assocoipahedra}\,(right) is also isomorphic to the pebble tree polytopes~$\pebbleTreePolytope[4][0,1]$ represented in \cref{fig:pebbleTreePolytopes} by combining Points~\eqref{prop:pebbleTreeComplexBalancingMap} and~\eqref{prop:pebbleTreeComplexRerootingMap} of~\cref{prop:pebbleTreeComplexOperations}.
\begin{figure}
	\centerline{\includegraphics[scale=.3]{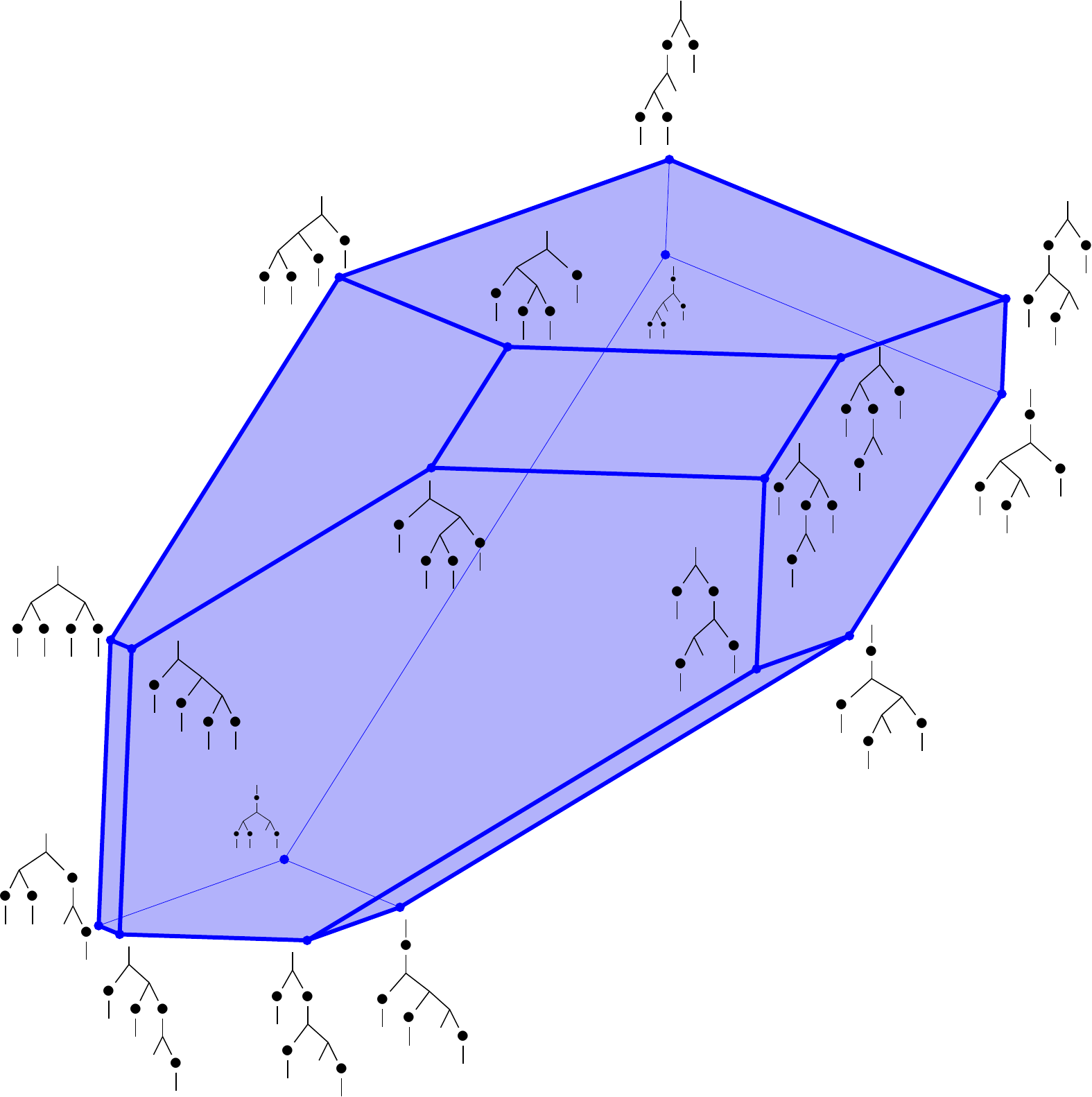} \hspace{.2cm} \includegraphics[scale=.3]{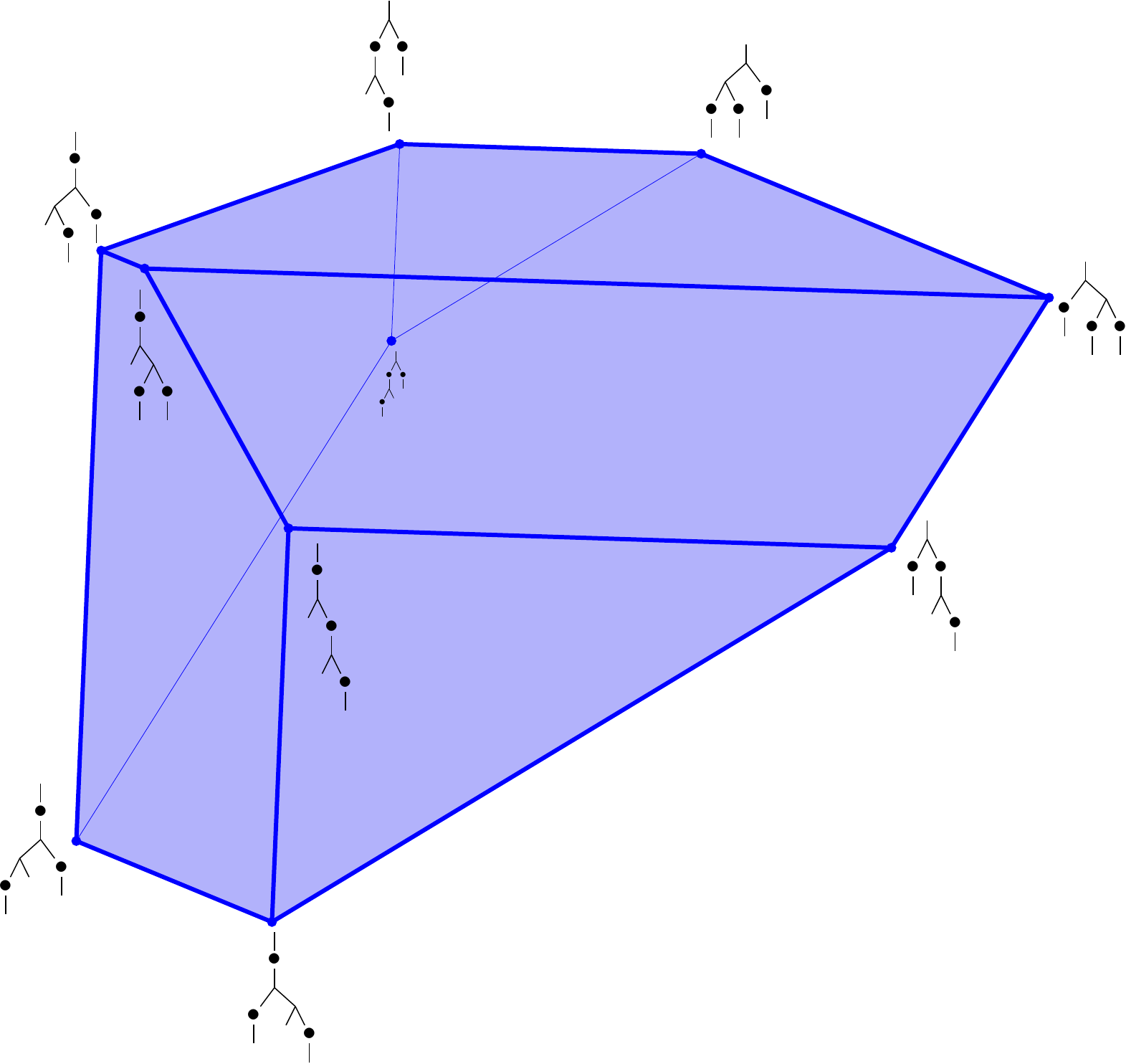}}
	\caption{The $\alpha$-assocoipahedra for $\alpha = \textsc{oiioi}$ (left) and~$\alpha = \textsc{oooi}$ (right), realized as faces of the pebble tree polytopes~$\pebbleTreePolytope[4][1,0]$ and~$\pebbleTreePolytope[3][1,0]$.}
	\label{fig:assocoipahedra}
\end{figure}
\end{remark}


\pagebreak

\section{Pebble tree numerology}
\label{sec:pebbleTreeNumerology}

In this section, we compute the generating functions of the maximal pebble trees (\cref{subsec:enumerationMPT}) and of all the pebble trees (\cref{subsec:enumerationPT}), and gather explicit expansions of these generating functions (\cref{subsec:generatingFunctions}).
While the methods are standard computations based on generatingfunctionology~\cite{FlajoletSedgewick}, the results reveal a few surprises.
All references like~\OEIS{A000108} are entries of the Online Encyclopedia of Integer Sequences~\cite{OEIS}.


\subsection{Enumeration of maximal pebble trees}
\label{subsec:enumerationMPT}

We start with the enumeration of the maximal pebble trees which is significantly simpler.

\begin{definition}
\label{def:generatingFunctionsMPT}
For~$\ell, u, b \in \N$, we denote by $m_\ell^{b,u}$ the number of maximal pebble trees of~$\pebbleTrees$ (\ie with $\ell$ leaves, $b$ balanced and $u$ unbalanced~colors).
We consider the generating functions
\[
\maximalPebbleTreeGeneratingFunction(x) \eqdef \sum_{\ell = 1}^\infty m_\ell^{b,u} \, x^\ell
\qquad\text{and}\qquad
\maximalPebbleTreeGeneratingFunction[k](x,y) \eqdef \sum_{u = 0}^k \maximalPebbleTreeGeneratingFunction[k-u,u](x) \, \frac{y^u}{u!} = \sum_{\ell = 1}^\infty \sum_{u = 0}^k m_\ell^{k-u,u} \, x^\ell \, \frac{y^u}{u!}.
\]
\end{definition}

\begin{proposition}
\label{prop:generatingFunctionsMPT}
The generating functions~$\maximalPebbleTreeGeneratingFunction(x)$ satisfy the functional equations
\[
\maximalPebbleTreeGeneratingFunction(x) = x \cdot \delta_{b = 0} + b \cdot  \maximalPebbleTreeGeneratingFunction[b-1,u+1](x) + \sum_{v = 0}^u \binom{u}{v} \cdot \maximalPebbleTreeGeneratingFunction[b+v,u-v](x) \cdot \maximalPebbleTreeGeneratingFunction[b+u-v,v](x),
\]
where~$\delta$ is the Kronecker delta.
Hence, $\maximalPebbleTreeGeneratingFunction(x)$ is algebraic for any~$b,u \in \N$.
Moreover, the generating function~$\maximalPebbleTreeGeneratingFunction[k](x,y)$ satisfies the differential equation
\[
\maximalPebbleTreeGeneratingFunction[k](x,y) = x \, \frac{y^k}{k!} + \Big( k \, \frac{\partial}{\partial y} - y \, \frac{\partial^2}{\partial y^2} \Big) \, \maximalPebbleTreeGeneratingFunction[k](x,y) + \maximalPebbleTreeGeneratingFunction[k](x,y)^2.
\]
\end{proposition}

\begin{proof}
A maximal pebble tree in~$\pebbleTrees[{[s,t]}][B,U]$ is:
\begin{itemize}
\item either a leaf if~$s = t$ and~$B = \varnothing$,
\item or a unary node with a single pebble colored by~$\pebbleColor \in B$ and a child in~$\pebbleTrees[{[s,t]}][B \ssm \{\pebbleColor\},U \cup \{\pebbleColor\}]$,
\item or a binary node with no pebble and two children in~$\pebbleTrees[{[s,r]}][B,V]$ and~$\pebbleTrees[{[r,t]}][B,U \ssm V]$ for some~$r \in [s,t]$ and~$V \subseteq U$.
\end{itemize}
The functional equations for~$\maximalPebbleTreeGeneratingFunction(x)$ are immediate consequences of this structural decomposition by classical generatingfunctionology~\cite{FlajoletSedgewick}.
The algebraicity of~$\maximalPebbleTreeGeneratingFunction(x)$ follows as it belongs to a system of finitely many polynomial equations (all equations for a given sum~$b+u$), and the functional equation for~$\maximalPebbleTreeGeneratingFunction[k](x,y)$ is obtained by classical manipulations on the functional equations for~$\maximalPebbleTreeGeneratingFunction(x)$.
\end{proof}

\begin{example}
When~$b = u = 0$, we recover the functional equation
\(
\maximalPebbleTreeGeneratingFunction[0,0](x) = x + \maximalPebbleTreeGeneratingFunction[0,0](x)^2
\)
which yields the classical Catalan generating function
\[
\maximalPebbleTreeGeneratingFunction[0,0](x) = \frac{1-\sqrt{1-4x}}{2} = 
x + x^2 + 2 x^3 + 5 x^4 + 14 x^5 + 42 x^6 + 132 x^7 + 429 x^8 + 1430 x^9 + \dots \hspace{.3cm} (\OEIS{A000108})
\]
\end{example}

\begin{example}
For $b + u = 1$, we obtain that
\[
\maximalPebbleTreeGeneratingFunction[1,0](x) = \maximalPebbleTreeGeneratingFunction[0,1](x) + \maximalPebbleTreeGeneratingFunction[1,0](x) ^2
\qquad\text{and}\qquad
\maximalPebbleTreeGeneratingFunction[0,1](x) = x + 2 \cdot \maximalPebbleTreeGeneratingFunction[1,0](x) \cdot \maximalPebbleTreeGeneratingFunction[0,1](x) 
\]
from which we can compute the expansions
\begin{align*}
\maximalPebbleTreeGeneratingFunction[1,0](x) & =  x + 3 x^2 + 16 x^3 + 105 x^4 + 768 x^5 + 6006 x^6 + 49152 x^7 + \dots \qquad (\OEIS{A085614})
\\
\maximalPebbleTreeGeneratingFunction[0,1](x) & =  x + 2 x^2 + 10 x^3 + 64 x^4 + 462 x^5 + 3584 x^6 + 29172 x^7 + \dots \qquad\;\; (\OEIS{A078531})
\end{align*}
These functions actually both satisfy a cubic equation, namely
\[
2 \cdot \maximalPebbleTreeGeneratingFunction[1,0](x)^3 - 3 \cdot \maximalPebbleTreeGeneratingFunction[1,0](x)^2 + \maximalPebbleTreeGeneratingFunction[1,0](x) - x = 0
\qquad\text{and}\qquad
4 \cdot \maximalPebbleTreeGeneratingFunction[0,1](x)^3 - \maximalPebbleTreeGeneratingFunction[0,1](x)^2 + x^2 = 0.
\]
\end{example}

\begin{example}
\label{exm:10Catalan}
For $b + u = 2$, we obtain
\begin{align}
\maximalPebbleTreeGeneratingFunction[2,0](x) & = 2 \cdot \maximalPebbleTreeGeneratingFunction[1,1](x) + \maximalPebbleTreeGeneratingFunction[2,0](x)^2,
\label{eq:maximalPebbleTreeGeneratingFunction20}
\\
\maximalPebbleTreeGeneratingFunction[1,1](x) & = \maximalPebbleTreeGeneratingFunction[0,2](x) + 2 \cdot \maximalPebbleTreeGeneratingFunction[2,0](x) \cdot \maximalPebbleTreeGeneratingFunction[1,1](x),
\label{eq:maximalPebbleTreeGeneratingFunction11}
\\
\maximalPebbleTreeGeneratingFunction[0,2](x) & = x + 2 \cdot \maximalPebbleTreeGeneratingFunction[2,0](x) \cdot \maximalPebbleTreeGeneratingFunction[0,2](x) + 2 \cdot \maximalPebbleTreeGeneratingFunction[1,1](x)^2.
\label{eq:maximalPebbleTreeGeneratingFunction02}
\end{align}
From which we can compute the expansions:
\begin{align*}
\maximalPebbleTreeGeneratingFunction[2,0](x) & =  2 x + 24 x^2 + 496 x^3 + 12560 x^4 + 354048 x^5 + 10665088 x^6 + \dots
\\
\maximalPebbleTreeGeneratingFunction[1,1](x) & = x + 10 x^2 + 200 x^3 + 5000 x^4 + 140000 x^5 + 4200000 x^6 +  \dots \qquad (\OEIS{A156275})
\\
\maximalPebbleTreeGeneratingFunction[0,2](x) & =  x + 6 x^2 + 112 x^3 + 2728 x^4 + 75360 x^5 + 2242304 x^6 + \dots
\end{align*}
The expansion of~$\maximalPebbleTreeGeneratingFunction[1,1](x)$ is quite surprising, but can be explained by a tiny functional miracle.
Indeed, observe that we obtain that $2 \cdot \maximalPebbleTreeGeneratingFunction[1,1](x) = \maximalPebbleTreeGeneratingFunction[2,0](x) \cdot \big(1 - \maximalPebbleTreeGeneratingFunction[2,0](x) \big)$ from \cref{eq:maximalPebbleTreeGeneratingFunction20} and that $\maximalPebbleTreeGeneratingFunction[0,2](x) = \maximalPebbleTreeGeneratingFunction[1,1](x) \cdot \big(1 - 2 \cdot \maximalPebbleTreeGeneratingFunction[2,0](x) \big)$ from \cref{eq:maximalPebbleTreeGeneratingFunction11}.
Replacing $\maximalPebbleTreeGeneratingFunction[0,2](x)$ on both sides of \cref{eq:maximalPebbleTreeGeneratingFunction02}, we obtain
\begin{equation}
\maximalPebbleTreeGeneratingFunction[1,1](x) = x + 4 \cdot \maximalPebbleTreeGeneratingFunction[1,1](x) \cdot \maximalPebbleTreeGeneratingFunction[2,0](x) \cdot \big( 1 - \maximalPebbleTreeGeneratingFunction[2,0](x) \big) + 2 \cdot \maximalPebbleTreeGeneratingFunction[1,1](x)^2 = x + 10 \cdot \maximalPebbleTreeGeneratingFunction[1,1](x)^2.
\label{eq:10Catalan}
\end{equation}
This shows that~$m_\ell^{1,1} = 10^{\ell-1} C_\ell$, where~$C_\ell = \frac{1}{\ell+1} \binom{2\ell}{\ell}$ is the Catalan number.

Consider the map sending a maximal pebble tree of~$\pebbleTrees[\ell][1,1]$ to its underlying binary tree.
In view of the formula~${m_\ell^{1,1} = 10^{\ell-1} C_\ell}$, it is natural to expect that its fibers all have size~$10^{\ell-1}$.
However, while the size of the fiber of a binary tree~$T$ is clearly invariant under reordering the children of~$T$, it is not constant on all binary trees already for~$\ell = 4$.
Namely, the fiber of the binary tree whose children are both the tree on $2$ leaves contains $968$ pebble trees of~$\pebbleTrees[4][1,1]$, while the fiber of each of the remaining $4$ binary trees on $4$ leaves contains $1008$ pebble trees of~$\pebbleTrees[4][1,1]$.

Despite this disappointing observation, one can provide a bijective explanation for the appearance of the Catalan numbers in~$m_\ell^{1,1}$.
It requires the following observation.

\parpic(2.5cm,.1cm)(-2pt, 60pt)[r][b]{
\begin{tabular}{c}
	\pebbleTree{[ [ [ $T_1$ ] [ [ \leaf ] [ [ \leaf ] [ [ \leaf ] [ [ $T_k$ ] [ $\circ$ [ $X$ ] ] ] ] ] ] ] ]} \\[-.5cm]
	\pebbleTree{[ [ [ $T_1$ ] [ [ \leaf ] [ [ \leaf ] [ [ \leaf ] [ [ $T_k$ ] [ [ $Y$ ] [ $Z$ ] ] ] ] ] ] ] ]}
\end{tabular}
}{
\noindent
\begin{minipage}{12.5cm}
\begin{remark}
\label{rem:decomposition}
Any maximal $\circ$-balanced $\bullet$-unbalanced pebble tree can be uniquely obtained from a right comb by attaching 
\begin{itemize}
\item to each left leaf, a maximal $\{\circ, \bullet\}$-balanced pebble tree,
\item to the right leaf, a maximal $\circ$-balanced $\bullet$-unbalanced pebble tree, whose root is
	\begin{itemize}
	\item either a unary node with a $\circ$ pebble, whose unique child~$X$ is $\{\circ, \bullet\}$-unbalanced,
	\item or a binary node with no pebble, whose left child~$Y$ is $\circ$-balanced $\bullet$-unbalanced and whose right child~$Z$ is $\{\circ, \bullet\}$-balanced.
	\end{itemize}
\end{itemize}
These two options are illustrated on the right.
A similar decomposition holds with a left comb instead of a right comb.
\end{remark}
\end{minipage}
\vspace{.2cm}
}

We now proceed to define, for~$\ell, r \ge 1$, an explicit bijection~$\psi$ sending a triple~$(T, L, R)$ of maximal pebble trees of~$\pebbleTrees[2][1,1]$, $\pebbleTrees[\ell][1,1]$ and $\pebbleTrees[r][1,1]$ respectively to a maximal pebble tree of~$\pebbleTrees[\ell + r][1,1]$.
The image~$\psi(T, L, R)$ is described in \cref{fig:bijection}.
Note that~$\psi(T, L, R)$ sometimes depends on the type of~$L$ or~$R$ in the sense of the decomposition of \cref{rem:decomposition}.
In this description, we denote by~$T_{\yinyang}$ the $\bullet$-balanced $\circ$-unbalanced pebble tree obtained by exchanging the $\circ$ and $\bullet$ pebbles in a $\circ$-balanced $\bullet$-unbalanced pebble tree~$T$.
As the decomposition of \cref{rem:decomposition} is unique, the map~$\psi$ is well defined, and it is immediate to check that the resulting trees are maximal pebble trees of~$\pebbleTrees[\ell + r][1,1]$.
Again by~\cref{rem:decomposition}, it is straightforward to check that $\psi$ is bijective.
Finally, as~$m_2^{1,1} = 10$, the existence of the bijection~$\psi$ directly implies Equation~\eqref{eq:10Catalan}, hence the fact that~${m_\ell^{1,1} = 10^{\ell-1} C_\ell}$.

\begin{figure}
	\centerline{
		\begin{tabular}{|c|@{\hspace{.48cm}}c@{\hspace{.48cm}}|@{\hspace{.48cm}}c@{\hspace{.48cm}}|@{\hspace{.48cm}}c@{\hspace{.48cm}}|@{\hspace{.44cm}}c@{\hspace{.44cm}}|@{\hspace{.44cm}}c@{\hspace{.44cm}}|@{\hspace{.44cm}}c@{\hspace{.44cm}}|}
			\hline
			\raisebox{-1cm}{$T$} &
      			\pebbleTree{[ [ [ $\circ$ [ \leaf ] ] [ $\bullet$ [ $\circ$ [ \leaf ] ] ] ] ]} &
      			\pebbleTree{[ [ [ $\bullet$ [ $\circ$ [ \leaf ] ] ] [ $\circ$ [ \leaf ] ] ] ]} &
      			\pebbleTree{[ [ [ $\circ$ [ \leaf ] ] [ $\circ$ [ $\bullet$ [ \leaf ] ] ] ] ]} &
      			\pebbleTree{[ [ [ $\circ$ [ $\bullet$ [ \leaf ] ] ] [ $\circ$ [ \leaf ] ] ] ]} &
      			\pebbleTree{[ [ $\circ$ [ [ $\circ$ [ \leaf ] ] [ $\bullet$ [ \leaf ] ] ] ] ]} &
      			\pebbleTree{[ [ $\circ$ [ [ $\bullet$ [ \leaf ] ] [ $\circ$ [ \leaf ] ] ] ] ]}
			\\
			\hline
			\raisebox{-.8cm}{$\psi(T, L, R)$} &
      			\pebbleTree{[ [ [ $L$ ] [ $\bullet$ [ $R$ ] ] ] ]} &
      			\pebbleTree{[ [ [ $\bullet$ [ $L$ ] ] [ $R$ ] ] ]} &
      			\pebbleTree{[ [ [ $L$ ] [ $\circ$ [ $R_{\yinyang}$ ] ] ] ]} &
      			\pebbleTree{[ [ [ $\circ$ [ $L_{\yinyang}$ ] ] [ $R$ ] ] ]} &
			\pebbleTree{[ [ $\circ$ [ [ $L$ ] [ $R_{\yinyang}$ ] ] ] ]} &
			\pebbleTree{[ [ $\circ$ [ [ $L_{\yinyang}$ ] [ $R$ ] ] ] ]}
			\\
			\hline
		\end{tabular}
	}
	\vspace{.1cm}
	\centerline{
		\begin{tabular}{|c|c|c|c|c|}
			\hline
			\raisebox{-1.3cm}{$T$} &
      			\multicolumn{2}{c|}{\pebbleTree{[ [ $\circ$ [ [ \leaf ] [ $\bullet$ [ $\circ$ [ \leaf ] ] ] ] ] ]}} &
      			\multicolumn{2}{c|}{\pebbleTree{[ [ $\circ$ [ [ \leaf ] [ $\circ$ [ $\bullet$ [ \leaf ] ] ] ] ] ]}}
			\\
			\hline
			\raisebox{-1.3cm}{$L$} &
			\pebbleTree{[ [ [ $T_1$ ] [ [ \leaf ] [ [ \leaf ] [ [ \leaf ] [ [ $T_k$ ] [ $\circ$ [ $X$ ] ] ] ] ] ] ] ]} &
			\pebbleTree{[ [ [ $T_1$ ] [ [ \leaf ] [ [ \leaf ] [ [ \leaf ] [ [ $T_k$ ] [ [ $Y$ ] [ $Z$ ] ] ] ] ] ] ] ]} &
			\pebbleTree{[ [ [ $T_1$ ] [ [ \leaf ] [ [ \leaf ] [ [ \leaf ] [ [ $T_k$ ] [ $\circ$ [ $X$ ] ] ] ] ] ] ] ]} &
			\pebbleTree{[ [ [ $T_1$ ] [ [ \leaf ] [ [ \leaf ] [ [ \leaf ] [ [ $T_k$ ] [ [ $Y$ ] [ $Z$ ] ] ] ] ] ] ] ]}
			\\[-.3cm]
			&&&&
			\\
			\hline
			\raisebox{-1.8cm}{$\psi(T, L, R)$} &
			\pebbleTree{[ [ $\circ$ [ [ $X$ ] [ [ $T_1$ ] [ [ \leaf ] [ [ \leaf ] [ [ \leaf ] [ [ $T_k$ ] [ $\bullet$ [ $R$ ] ] ] ] ] ] ] ] ] ]} &
			\pebbleTree{[ [ [ $Y$ ] [ [ $T_1$ ] [ [ \leaf ] [ [ \leaf ] [ [ \leaf ] [ [ $T_k$ ] [ [ $Z$ ] [ $\bullet$ [ $R$ ] ] ] ] ] ] ] ] ] ]} &
			\pebbleTree{[ [ $\circ$ [ [ $X$ ] [ [ $T_1$ ] [ [ \leaf ] [ [ \leaf ] [ [ \leaf ] [ [ $T_k$ ] [ $\circ$ [ $R_{\yinyang}$ ] ] ] ] ] ] ] ] ] ]} &
			\pebbleTree{[ [ [ $Y$ ] [ [ $T_1$ ] [ [ \leaf ] [ [ \leaf ] [ [ \leaf ] [ [ $T_k$ ] [ [ $Z$ ] [ $\circ$ [ $R_{\yinyang}$ ] ] ] ] ] ] ] ] ] ]}
			\\
			\hline
		\end{tabular}
	}
	\vspace{.1cm}
	\centerline{
		\begin{tabular}{|c|c|c|c|c|}
			\hline
			\raisebox{-1.3cm}{$T$} &
      			\multicolumn{2}{c|}{\pebbleTree{[ [ $\circ$ [ [ $\bullet$ [ $\circ$ [ \leaf ] ] ] [ \leaf ] ] ] ]}} &
      			\multicolumn{2}{c|}{\pebbleTree{[ [ $\circ$ [ [ $\circ$ [ $\bullet$ [ \leaf ] ] ] [ \leaf ] ] ] ]}}
			\\
			\hline
			\raisebox{-1.3cm}{$R$} &
			\pebbleTree{[ [ [ [ [ [ [ $\circ$ [ $X$ ] ] [ $T_k$ ] ] [ \leaf ] ] [ \leaf ] ] [ \leaf ] ] [ $T_1$ ] ] ]} &
			\pebbleTree{[ [ [ [ [ [ [ [ $Z$ ] [ $Y$ ] ] [ $T_k$ ] ] [ \leaf ] ] [ \leaf ] ] [ \leaf ] ] [ $T_1$ ] ] ]} &
			\pebbleTree{[ [ [ [ [ [ [ $\circ$ [ $X$ ] ] [ $T_k$ ] ] [ \leaf ] ] [ \leaf ] ] [ \leaf ] ] [ $T_1$ ] ] ]} &
			\pebbleTree{[ [ [ [ [ [ [ [ $Z$ ] [ $Y$ ] ] [ $T_k$ ] ] [ \leaf ] ] [ \leaf ] ] [ \leaf ] ] [ $T_1$ ] ] ]}
			\\[-.3cm]
			&&&&
			\\
			\hline
			\raisebox{-1.8cm}{$\psi(T, L, R)$} &
			\pebbleTree{[ [ $\circ$ [ [ [ [ [ [ [ $\bullet$ [ $L$ ] ] [ $T_k$ ] ] [ \leaf ] ] [ \leaf ] ] [ \leaf ] ] [ $T_1$ ] ] [ $X$ ] ] ] ]} &
			\pebbleTree{[ [ [ [ [ [ [ [ [ $\bullet$ [ $L$ ] ] [ $Z$ ] ] [ $T_k$ ] ] [ \leaf ] ] [ \leaf ] ] [ \leaf ] ] [ $T_1$ ] ] [ $Y$ ] ] ]} &
			\pebbleTree{[ [ $\circ$ [ [ [ [ [ [ [ $\circ$ [ $L_{\yinyang}$ ] ] [ $T_k$ ] ] [ \leaf ] ] [ \leaf ] ] [ \leaf ] ] [ $T_1$ ] ] [ $X$ ] ] ] ]} &
			\pebbleTree{[ [ [ [ [ [ [ [ [ $\circ$ [ $L_{\yinyang}$ ] ] [ $Z$ ] ] [ $T_k$ ] ] [ \leaf ] ] [ \leaf ] ] [ \leaf ] ] [ $T_1$ ] ] [ $Y$ ] ] ]}
			\\
			\hline
		\end{tabular}
	}
	\caption{Bijection~$\psi$ sending a triple~$(T, L, R)$ of maximal pebble trees of~$\pebbleTrees[2][1,1]$, $\pebbleTrees[\ell][1,1]$ and $\pebbleTrees[r][1,1]$ respectively to a maximal pebble tree of~$\pebbleTrees[\ell + r][1,1]$.}
	\label{fig:bijection}
\end{figure}
\end{example}



\subsection{Enumeration of all pebble trees}
\label{subsec:enumerationPT}

We now consider all pebble trees.

\begin{definition}
\label{def:generatingFunctionPT}
For~$\ell, n, u, b \in \N$, we denote by $p_{\ell,n}^{b,u}$ the number of pebble tree with $\ell$ leaves, $n$ nodes, $b$ balanced and $u$ unbalanced colors.
We consider the generating function
\[
\pebbleTreeGeneratingFunction(x,y) \eqdef \sum_{\ell = 1}^\infty \sum_{n = 0}^\infty p_{\ell,n}^{b,u} \, x^\ell \, y^n.
\]
\end{definition}

\begin{proposition}
\label{prop:generatingFunctionPT}
The generating functions~$\pebbleTreeGeneratingFunction(x,y)$ satisfy the functional equations
\[
\pebbleTreeGeneratingFunction(x,y) = x \cdot \delta_{b = 0} + y \cdot \sum_{s =1}^b \binom{b}{s} \cdot \pebbleTreeGeneratingFunction[b-s,u+s](x,y) + y \cdot \!\!\!\!\! \sum_{\substack{d \ge 2 \\ X_1, \dots, X_b \\ Y_1, \dots, Y_u}}  \prod_{k = 1}^d \pebbleTreeGeneratingFunction[b_k, u_k](x,y)
\]
where~$\delta$ is the Kronecker delta, where each $X_i$ ranges among arbitrary subsets of~$[d]$ while each~$Y_j$ ranges among strict subsets of~$[d]$, and where~$b_k \eqdef |\set{i \in [b]}{k \in X_i}| + |\set{j \in [u]}{k \in Y_j}|$ and $u_k \eqdef b+u-b_k$ for any~$k \in [d]$.
Hence, $\pebbleTreeGeneratingFunction(x)$ is algebraic for any~$b,u \in \N$.
\end{proposition}

\begin{proof}
A pebble tree of~$\pebbleTrees[{[s,t]}][B,U]$ is:
\begin{itemize}
\item either a leaf if~$s = t$ and~$B = \varnothing$,
\item or a unary node with some pebbles colored by a non-empty subset~$S \subseteq B$ (one pebble of each color in~$S$) and a child in~$\pebbleTrees[{[s,t]}][B \ssm S,U \cup S]$,
\item or a node with some pebbles and~$d \ge 2$ children in~$\pebbleTrees[{[s_1,t_1]}][B_1,U_1], \dots, \pebbleTrees[{[s_d,t_d]}][B_d,U_d]$ respectively, for some~$s = s_1 \le t_1 = s_2 \le \dots \le t_d = t$ and~$B_i \subseteq B$ and~$U_i \subseteq U$ for all~$i \in [d]$, such that~$\bigcap_{i \in [d]} U_i = \varnothing$.
\end{itemize}
The functional equations for~$\pebbleTreeGeneratingFunction(x)$ are immediate consequences of this structural decomposition by classical generatingfunctionology~\cite{FlajoletSedgewick}.
The algebraicity of~$\pebbleTreeGeneratingFunction(x)$ follows as it belongs to a system of finitely many polynomial equations (all equations for a given sum~$b+u$).
\end{proof}

\begin{example}
When~$b = u = 0$, we recover the functional equation
\[
\pebbleTreeGeneratingFunction[0,0](x,y) = x + \frac{y \cdot \pebbleTreeGeneratingFunction[0,0](x,y)^2}{1-\pebbleTreeGeneratingFunction[0,0](x,y)}
\]
which yield the classical Schr\"oder generating function
\begin{align*}
\pebbleTreeGeneratingFunction[0,0](x,y) 
& = \frac{y \big( x+y-\sqrt{x^2-2xy-4xy^2+y^2} \big)}{2(1+y)} \\
& = x + x^2 y + x^3 (y + 2 y^2) + x^4 (y + 5 y^2 + 5 y^3) + x^5 (y + 9 y^2 + 21 y^3 + 14 y^4) + \dots
\end{align*}
\end{example}

The expansions of the generating functions~$\pebbleTreeGeneratingFunction[b,u](x,1)$ and~$\pebbleTreeGeneratingFunction[b,u](x,y)$ for~$b + u \le 2$ can be found in \cref{subsec:generatingFunctions}.
Finally, we observe that the evaluations of~$\pebbleTreeGeneratingFunction(x,y)$ at~$y = 1$ and~$y = -1$ have a geometric meaning.

\begin{proposition}
For any~$b,u \in \N$, the evaluation~$\pebbleTreeGeneratingFunction(x,1)$ is the generating function of the total number of faces of the pebble tree polytope~$\pebbleTreePolytope$, and 
\[
\pebbleTreeGeneratingFunction(x,-1) = \frac{x}{(-1)^b + (-1)^u x}.
\]
\end{proposition}

\begin{proof}
By \cref{thm:pebbleTreePolytope1}, $p_{\ell,n}^{b,u}$ is the number of $(\ell + \ell b + \ell u - u - 2 - n)$-dimensional faces of the $(\ell + \ell b + \ell u - u - 2)$-dimensional pebble tree polytope~$\pebbleTreePolytope$.
This implies that
\begin{itemize}
\item $\sum_n p_{\ell,n}^{b,u}$ is the total number of faces of~$\pebbleTreePolytope$,
\item $\sum_n p_{\ell,n}^{b,u} (-1)^n = (-1)^{\ell + \ell b + \ell u - u - 2}$ by Euler's formula.
\end{itemize}
This immediately implies the statement.
\end{proof}


\subsection{Expansions of generating functions}
\label{subsec:generatingFunctions}

Below are the expansions of the generating functions~$\maximalPebbleTreeGeneratingFunction[b,u](x)$, $\pebbleTreeGeneratingFunction[b,u](x,1)$ and~$\pebbleTreeGeneratingFunction[b,u](x,y)$ of \cref{def:generatingFunctionsMPT,def:generatingFunctionPT} for all $b + u \le 2$.

\para{$b = 0$ and $u = 0$}
\begin{align*}
\maximalPebbleTreeGeneratingFunction[0,0](x) & = x + x^2 + 2 x^3 + 5 x^4 + 14 x^5 + 42 x^6 + 132 x^7 + 429 x^8 + 1430 x^9 + 4862 x^10 + \dots \hspace{1.7cm} (\OEIS{A000108})
\\
\pebbleTreeGeneratingFunction[0,0](x,1) & = x + x^2 + 3 x^3 + 11 x^4 + 45 x^5 + 197 x^6 + 903 x^7 + 4279 x^8 + 20793 x^9 + 103049 x^{10} + \dots \hspace{.7cm} (\OEIS{A001003})
\\
\pebbleTreeGeneratingFunction[0,0](x,y) & = x + x^2 y + x^3 (y + 2 y^2) + x^4 (y + 5 y^2 + 5 y^3) + x^5 (y + 9 y^2 + 21 y^3 + 14 y^4) + \dots
\end{align*}

\para{$b = 1$ and $u = 0$}
\begin{align*}
\maximalPebbleTreeGeneratingFunction[1,0](x) & =  x + 3 x^2 + 16 x^3 + 105 x^4 + 768 x^5 + 6006 x^6 + 49152 x^7 + 415701 x^8 + 3604480 x^9 + \dots \hspace{.75cm} (\OEIS{A085614})
\\
\pebbleTreeGeneratingFunction[1,0](x,1) & = x + 7 x^2 + 81 x^3 + 1151 x^4 + 18225 x^5 + 308519 x^6 + 5465313 x^7 + 100051903 x^8 + \dots
\\
\pebbleTreeGeneratingFunction[1,0](x,y) & = x y + x^2 (y + 3 y^2 + 3 y^3) + x^3 (y + 8 y^2 + 24 y^3 + 32 y^4 + 16 y^5) + \dots
\end{align*}

\para{$b = 0$ and $u = 1$}
\begin{align*}
\maximalPebbleTreeGeneratingFunction[0,1](x) & =  x + 2 x^2 + 10 x^3 + 64 x^4 + 462 x^5 + 3584 x^6 + 29172 x^7 + 245760 x^8 + 2124694 x^9 + \dots \hspace{.9cm} (\OEIS{A078531})
\\
\pebbleTreeGeneratingFunction[0,1](x,1) & = x + 3 x^2 + 33 x^3 + 459 x^4 + 7185 x^5 + 120771 x^6 + 2129169 x^7 + 38843307 x^8 + \dots
\\
\pebbleTreeGeneratingFunction[0,1](x,y) & = x + x^2 (y + 2 y^2) + x^3 (y + 7 y^2 + 15 y^3 + 10 y^4) + x^4 (y + 14 y^2 + 68 y^3 + 152 y^4 + 160 y^5 + 64 y^6) + \dots
\end{align*}

\para{$b = 2$ and $u = 0$}
\begin{align*}
\maximalPebbleTreeGeneratingFunction[2,0](x) & =  2 x + 24 x^2 + 496 x^3 + 12560 x^4 + 354048 x^5 + 10665088 x^6 + 336114176 x^7 + 10237958656 x^8 + \dots
\\
\pebbleTreeGeneratingFunction[2,0](x,1) & = 3 x + 115 x^2 + 7431 x^3 + 587591 x^4 + 51702219 x^5 + 4860786491 x^6 + 478068368655 x^7 + \dots
\\
\pebbleTreeGeneratingFunction[2,0](x,y) & = x (y + 2 y^2) + x^2 (y + 9 y^2 + 33 y^3 + 48 y^4 + 24 y^5) + \dots
\end{align*}

\para{$b = 1$ and $u = 1$}
\begin{align*}
\maximalPebbleTreeGeneratingFunction[1,1](x) & =  x + 10 x^2 + 200 x^3 + 5000 x^4 + 140000 x^5 + 4200000 x^6 + 132000000 x^7 + \dots \hspace{2.5cm} (\OEIS{A156275})
\\
\pebbleTreeGeneratingFunction[1,1](x,1) & = x + 33 x^2 + 2061 x^3 + 160797 x^4 + 14049369 x^5 + 1315182201 x^6 + 128977070373 x^7 + \dots 
\\
\pebbleTreeGeneratingFunction[1,1](x,y) & = x y + x^2 (y + 7 y^2 + 15 y^3 + 10 y^4) + x^3 (y + 18 y^2 + 124 y^3 + 412 y^4 + 706 y^5 + 600 y^6 + 200 y^7) + \dots
\end{align*}

\para{$b = 0$ and $u = 2$}
\begin{align*}
\maximalPebbleTreeGeneratingFunction[0,2](x) & =  x + 6 x^2 + 112 x^3 + 2728 x^4 + 75360 x^5 + 2242304 x^6 + 70084864 x^7 + 2268770688 x^8 + \dots
\\
\pebbleTreeGeneratingFunction[0,2](x,1) & = x + 13 x^2 + 765 x^3 + 58297 x^4 + 5031129 x^5 + 467426661 x^6 + 45606874581 x^7 + \dots
\\
\pebbleTreeGeneratingFunction[0,2](x,y) & = x + x^2 (y + 6 y^2 + 6 y^3) + x^3 (y + 17 y^2 + 101 y^3 + 254 y^4 + 280 y^5 + 112 y^6) + \dots
\end{align*}


\section*{Aknowledgements}

I am grateful to Doriann Albertin for discussions around \cref{exm:10Catalan} and for a very constructive proofreading, to Wenjie Fang for pointing out \cref{rem:CoriJacquardSchaeffer} and for discussions on \cref{prop:generatingFunctionsMPT,prop:generatingFunctionPT}, to Guillaume Laplante-Anfossi for motivating me to find polytopal realizations of the assocoipahedra, to Thibaut Mazuir for algebraic discussions on the assocoipahedra, and to Anna de Mier, Jordi Castellví Foguet and Clément Requilé for discussions on \cref{exm:10Catalan} during the \href{https://gapcomb.upc.edu/en/seminar-en/gapcomb-workshop/3rd-gapcomb-workshop}{3rd Workshop on Geometric, Algebraic and Probabilistic Combinatorics} in Monserrat on July 2022.
This research benefited from computations and tests done using the open-source mathematical software \texttt{Sage}~\cite{Sage} and its combinatorics features developed by the \texttt{Sage-combinat} community~\cite{SageCombinat}, and from the Online Encyclopedia of Integer Sequences~\cite{OEIS}.


\bibliographystyle{alpha}
\bibliography{pebbleTrees}
\label{sec:biblio}

\end{document}